\documentclass[letterpaper,11pt,reqno]{amsart}

\usepackage{amsmath,amssymb,amscd,amsthm,amsxtra}

\usepackage[implicit=true]{hyperref}
\usepackage{bbm}
\allowdisplaybreaks

\usepackage{tikz}
\usetikzlibrary{matrix} 

\usepackage{enumitem}

\allowdisplaybreaks[2]

\usepackage{color}

\newtheorem{theorem}{Theorem} [section]

\newtheorem{lemma}[theorem]{Lemma}
\newtheorem{proposition}[theorem]{Proposition}
\newtheorem{remark}[theorem]{Remark}

\newtheorem{definition}[theorem]{Definition}

\newtheorem*{ack}{Acknowledgments}

\DeclareMathOperator*{\intt}{\int}

\DeclareMathOperator*{\supp}{supp}

\newcommand{\I}{\hspace{0.5mm}\text{I}\hspace{0.5mm}}
\newcommand{\II}{\text{I \hspace{-2.8mm} I} }
\newcommand{\III}{\text{I \hspace{-2.9mm} I \hspace{-2.9mm} I}}
\newcommand{\IV}{\text{I \hspace{-2.8mm} V} }

\newcommand{\noi}{\noindent}
\newcommand{\Z}{\mathbb{Z}}
\newcommand{\R}{\mathbb{R}}
\newcommand{\C}{\mathbb{C}}
\newcommand{\T}{\mathbb{T}}

\newcommand{\1}{\mathbbm{1}}

\let\Re=\undefined\DeclareMathOperator*{\Re}{Re}
\let\Im=\undefined\DeclareMathOperator*{\Im}{Im}

\let\P= \undefined
\newcommand{\P}{\mathbb{P}}
\newcommand{\PP}{\mathbf{P}}

\newcommand{\N}{\mathcal{N}}
\newcommand{\NB}{\mathbb{N}}

\newcommand{\FL}{\mathcal{F}L} 

\newcommand{\F}{\mathcal{F}}

\newcommand{\Gg}{\mathcal{G}}

\def\norm#1{\|#1\|}

\newcommand{\al}{\alpha}
\newcommand{\be}{\beta}
\newcommand{\dl}{\delta}

\DeclareMathOperator{\haar}{m}

\newcommand{\eps}{\varepsilon}

\newcommand{\s}{\sigma}

\newcommand{\ft}{\widehat}
\newcommand{\Ft}{{\mathcal{F}}}
\newcommand{\wt}{\widetilde}

\newcommand{\cj}{\overline}
\newcommand{\dx}{\partial_x}
\newcommand{\dt}{\partial_t}

\newcommand{\embeds}{\hookrightarrow}

\newcommand{\conj}[1]{\overline{#1}}

\renewcommand{\l}{\ell}

\newcommand{\les}{\lesssim}
\newcommand{\ges}{\gtrsim}

\newcommand{\jb}[1]
{\langle #1 \rangle}

\usepackage{bbm}

\DeclareMathOperator{\Id}{Id}

\numberwithin{equation}{section}
\numberwithin{theorem}{section}

\usepackage{todonotes}

\title[Invariance of the Gibbs measures for gKdV]{Invariance of the Gibbs measures for periodic generalized Korteweg-de Vries equations}
\date{\today}

\begin{document}
\baselineskip = 14pt

\author[A.~Chapouto]{Andreia Chapouto}
\address{
	Andreia Chapouto\\ School of Mathematics\\
	The University of Edinburgh\\
	and The Maxwell Institute for the Mathematical Sciences\\
	James Clerk Maxwell Building\\
	The King's Buildings\\
	Peter Guthrie Tait Road\\
	Edinburgh\\ 
	EH9 3FD\\United Kingdom}
\email{andreia.chapouto@ed.ac.uk}

\author[N.~Kishimoto]{Nobu Kishimoto}
\address{
	Nobu Kishimoto\\
	Research Institute for Mathematical Sciences\\
	Kyoto University\\
	Kitashirakawa Oiwake-cho\\
	Sakyo-ku\\
	Kyoto\\
	606-8502\\
	Japan}
\email{nobu@kurims.kyoto-u.ac.jp}




\begin{abstract}
	
	In this paper, we study the Gibbs measures for periodic generalized Korteweg-de Vries equations (gKdV) with quartic or higher nonlinearities. In order to bypass the analytical ill-posedness of the equation in the Sobolev support of the Gibbs measures, we establish deterministic well-posedness of the gauged gKdV equations within the framework of the Fourier-Lebesgue spaces. Our argument relies on bilinear and trilinear Strichartz estimates adapted to the Fourier-Lebesgue setting. Then, following Bourgain's invariant measure argument, we construct almost sure global-in-time dynamics and show invariance of the Gibbs measures for the gauged equations. These results can be brought back to the ungauged side by inverting the gauge transformation and exploiting the invariance of the Gibbs measures under spatial translations. We thus complete the program initiated by Bourgain (1994) on the invariance of the Gibbs measures for periodic gKdV equations.

\end{abstract}

\maketitle

\tableofcontents

\section{Introduction}
We study the Cauchy problem for the generalized Korteweg-de Vries equation (gKdV) on the one-dimensional torus $\T=\R / \Z$:
\begin{equation}\label{gkdv}
\begin{cases}
\partial_t u + \partial_x^3 u = \pm \dx (u^{k}), \\
u\vert_{t=0} = u_0,
\end{cases}
\quad (t,x)\in\R\times\T,
\end{equation}
where $k\geq 2$ is an integer. When $k=2$ and $k=3$, \eqref{gkdv} corresponds to the well-known Korteweg-de Vries (KdV) and modified Korteweg-de Vries (mKdV) equations, respectively. These two equations are known to be completely integrable, therefore satisfying infinitely many conservation laws, which is no longer true for \eqref{gkdv} with $k\geq4$.

The gKdV equation \eqref{gkdv} can be reformulated as a Hamiltonian system
\begin{equation*}
\dt u = \dx \frac{\delta H}{ \delta u}, \label{hamiltonian}
\end{equation*}
where $\frac{\delta H}{\delta u}$ denotes the Fr\'echet derivative and the Hamiltonian  is given by
\begin{align*}
H(u) : = \frac12 \int_\T (\dx u)^2 \,dx \pm \frac{1}{k+1} \int_\T u^{k+1} \, dx.
\end{align*}
In particular, $H(u)$ is conserved under the dynamics of \eqref{gkdv}. Note that the mean $\intt_\T u \, dx$ and the mass $M(u) = \intt_\T u^2 \, dx$ are also conserved quantities. Due to the conservation of the mean, we will restrict our discussion to mean zero initial data. See Remark~\ref{rm:mean} for further details on the non-zero mean case.

In view of the Hamiltonian structure of gKdV \eqref{gkdv}, we expect the Gibbs measure $\mu$ formally defined by
\begin{equation}
d \mu = Z^{-1} e^{-H(u)} \, du  = Z^{-1} e^{\mp\frac{1}{k+1}\int_\T u^{k+1}dx} e^{-\frac12\int_\T (\partial_x u )^2 dx} du, \label{measure}
\end{equation}
to be invariant under the dynamics of gKdV \eqref{gkdv}. In this paper, we complete the program initiated by Bourgain in \cite{BO94} by establishing invariance of the Gibbs measure $\mu$ in \eqref{measure} (under suitable normalization) for any $k \geq4$. Our result builds upon the work of Bourgain for KdV ($k=2$) and mKdV ($k=3$) \cite{BO93, BO94}, and of Richards \cite{R} for quartic gKdV ($k=4$).

The construction of Gibbs measures for Hamiltonian PDEs was initiated by Lebowitz-Rose-Speer \cite{LRS} in the context of the nonlinear Schr\"{o}dinger equation and has since been successfully pursued for other equations, see \cite{BO94, BO96, BO97, Tz06, Tzv08, BurqTzv07, BurqTzv08, OhKdv09, OhSBO09, Tzv10, ThTzv10, NahOhBelletSta12, BurqThoTzv13, BoBul14, Deng15, R, ORT16, OT18, DNY19, ORSW20, DNY21, GOTW} and references therein. The expression in \eqref{measure} is only formal, but it can be made rigorous by interpreting the Gibbs measure $\mu$ as a probability measure which is absolutely continuous with respect to the Gaussian measure~$\rho$ 
\begin{equation}\label{gaussian}
d\rho = Z_0^{-1} e^{-\frac12\int_\T (\dx u)^2 dx} du.
\end{equation}
The measure $\rho$ can be seen as the induced probability measure under the map
\begin{equation}
\omega \mapsto u^\omega(x) = \sum_{n\in \Z_*} \frac{g_n(\omega)}{|n|} e^{inx}, \label{induced}
\end{equation}
where $\{g_n\}_{n\in\Z_*}$, $\Z_* = \Z\backslash  \{0\}$, is a sequence of complex-valued independent Gaussian random variables on a probability space $(\Omega, \mathcal{F}, \P)$, satisfying $g_{-n} = \conj{g_n}$. Note that $u$ defined in \eqref{induced} lies in $\bigcap_{s<\frac12} H^s(\T)$ almost surely. Consequently, the support of $\rho$ and of $\mu$ (when well-defined) is included in this set.

In order to discuss the invariance of the Gibbs measure $\mu$, we must first construct a (globally-in-time) well-defined flow for gKdV \eqref{gkdv} on the support of $\mu$. Before proceeding, we recall some known well-posedness results of \eqref{gkdv}. 
In \cite{BO93}, Bourgain introduced the Fourier restriction norm method and proved local well-posedness of KdV in $L^2(\T) \supset \supp\mu$, which was immediately extended to global well-posedness due to the conservation of mass. Following the same method, for mKdV Bourgain \cite{BO94} established its local well-posedness in $H^{s_1}(\T) \cap \FL^{s_2,\infty}(\T) \supset \supp \mu$ for some $s_1<\frac12 < s_2<1$, where $\FL^{s,p}(\T)$ denotes the Fourier-Lebesgue space defined through the norm
\begin{equation}
\| f\|_{\FL^{s,p}} = \| \jb{n}^s \ft{f}(n) \|_{\l^p_n}, \label{FL}
\end{equation}
where $\jb{\cdot} = (1 + |\cdot|^2)^\frac12$ and $\ft{f}(n)$ denotes the Fourier coefficient of $f$. In fact, it was shown by B\'enyi-Oh \cite{BO2011} that the function $u$ defined in \eqref{induced} lies in $ \FL^{s,p}(\T)$ almost surely if $(s-1)p<-1$, and thus these spaces used in \cite{BO94} also include the support of the measure~$\mu$.  Unfortunately, the conservation laws of mKdV were not sufficient to globalize solutions.%
	\footnote{After \cite{BO94}, local well-posedness for mKdV on $\T$ has been obtained in less regular Sobolev and Fourier-Lebesgue spaces including the support of the Gibbs measure.
Furthermore, the local-in-time flow constructed on these spaces has been globalized (in the deterministic sense) by exploiting complete integrability of mKdV. For details, see, e.g., the recent works \cite{Ch21,Ch} by the first author and references therein.}
	Instead, Bourgain used a probabilistic argument to construct global-in-time solutions of mKdV. In the seminal work \cite{BO94}, he exploited the invariance of the finite dimensional Gibbs measures corresponding to the truncated dynamics to globalize solutions of mKdV. Moreover, he rigorously established the invariance of the Gibbs measure $\mu$ for KdV and mKdV. 
Here, with the solution map $\Psi (t)$ of \eqref{gkdv} given at least almost surely with respect to $\mu$, invariance of $\mu$ is understood as 
\begin{align}\label{intro_invariance}
	\mu\big( \Psi(-t) A\big) = \mu(A),
\end{align}
for any measurable set $A$ and $t\in\R$.
This approach is known as Bourgain's invariant measure argument. The main breakthrough in \cite{BO94} was the globalization argument, in particular, using the formal invariance of the Gibbs measure $\mu$ as a substitute for a conservation law. 


Regarding \eqref{gkdv} with $k\geq4$, in \cite{BO93}, Bourgain proved local existence of solutions (without uniqueness) in $H^s(\T)$ for $s \geq 1$. Later, Staffilani \cite{ST} upgraded this result to local well-posedness in $H^s(\mathbb{T})$, $s\geq 1$ (and hence global well-posedness under the presence of a priori $H^1$-control), which extended the former local well-posedness result for $s>\frac32$ obtained by the classical approach not exploiting dispersion (see Kato's results \cite{K75,K79} on quasilinear hyperbolic systems).
In \cite{CKSTT04}, Colliander-Keel-Staffilani-Takaoka-Tao further refined the relevant nonlinear estimate in the Fourier restriction norm to establish local well-posedness in $H^s(\T)$ for $s\geq \frac12$ and used the $I$-method to construct global solutions for $s>\frac56$ and $k=4$ (see also \cite{HuLi13,BaoWu17}). 
The common strategy in \cite{BO93,ST,CKSTT04} is to study the following gauged gKdV equation ($\mathcal{G}$-gKdV):
\begin{equation}
\dt u + \dx^3 u = \pm \dx\big( u^{k} - k\PP_0 (u^{k-1})u   \big), \label{gauged}
\end{equation}
where $\PP_0$ denotes the mean $\PP_0(f) = \int_\T f \, dx$. In fact, compared with the original equation \eqref{gkdv}, certain problematic frequency interactions have been removed in the nonlinearity of the gauged equation \eqref{gauged}.
Note that the two equations \eqref{gauged} and \eqref{gkdv} are equivalent\footnote{The equivalence is obvious for smooth solutions. Even for $u\in C([-T,T];\mathcal{F}L^{s,p}(\T))$, we can show the equivalence when $s\geq 0$ and $s+\frac{1}{p}>1-\frac{1}{k}$. In fact, since the embedding $\mathcal{F}L^{s,p}(\T)\hookrightarrow L^k(\T)$ holds in this case, the gauge transformation \eqref{gauge} can be defined in the sense of spatial translation of an $L^k$-function for each $t$, and it is continuous on $C([-T,T];\mathcal{F}L^{s,p}(\T))$, as we will show in Lemma~\ref{lem:cont-G}. Note also that both of the nonlinearities $\pm \dx ( u^{k} )$ and $\pm \dx (u^k - k\PP_0 (u^{k-1})u)$ make sense as distributions in $C([-T,T];\mathcal{F}L^{-1-\frac{1}{p}-,p}(\T))$ in view of the embedding $L^1\hookrightarrow \mathcal{F}L^{-\frac{1}{p}-,p}(\T)$. Then, it is not hard to verify that $u\in C([-T,T];\mathcal{F}L^{s,p}(\T))$ is a mild solution of \eqref{gkdv} (i.e., $u$ satisfies the associated integral equation in $\mathcal{F}L^{-1-\frac{1}{p}-,p}(\T)$ pointwise in $t$) if and only if $v=\mathcal{G}_{0,t}u$ is a mild solution of \eqref{gauged}. In Theorem~\ref{th:lwp} below, we only consider $s,p$ within this range.} in the following sense: $u$ is a solution of \eqref{gkdv} if and only if $v = \mathcal{G}(u)$ is a solution of \eqref{gauged}, where the gauge transformation\footnote{Although $\mathcal{G}$ is a spatial translation only when considered as a map on space-time functions, we intend to follow the literature \cite{CKSTT04, R, ORT16} and call it a gauge transformation. We also refer to the transformed equation \eqref{gauged} as the gauged equation. See Remark~\ref{rm:DNLS}(ii) for further discussion.} $\mathcal{G} = \mathcal{G}_{0,t}$ is given by 
\begin{equation}\label{gauge}
 \mathcal{G}_{0,t} (u) (t,x) =  u\Big(t, x \mp k \int_0^t \PP_0\big( u^{k-1}(t') \big)  \, dt' \Big).
\end{equation}

Concerning the construction of invariant Gibbs dynamics, we first recall that the Gibbs measure is supported in $H^s(\T)$ for $s<\frac12$. However, this range is exactly where the (gauged) gKdV equation on $\T$ is known to be analytically ill-posed.
In fact, it has been shown in \cite{BO97,CCT03, CKSTT04} that the data-to-solution map fails to be $C^k$-continuous, which means that one cannot use a contraction mapping argument to construct the flow. 
To bypass this difficulty, Richards \cite{R}, following the argument in \cite{BO94, CO12}, established probabilistic local well-posedness of \eqref{gkdv} with $k=4$ in $H^s(\T)$ for $s<\frac12$, and proved invariance of the Gibbs measure under the flow of \eqref{gauged}. 
To the knowledge of the authors, the question of invariance of $\mu$ in the sense of \eqref{intro_invariance} for $k\geq 5$ remains open (see Remark~\ref{rm:ORT}(ii) for a result on invariance in a weaker sense).

The main difficulty in applying the probabilistic approach in \cite{R} to gKdV with higher order nonlinearities is that the number of cases involved in analyzing the nonlinearity increases with its degree. 
It may be possible to bypass this difficulty and implement a probabilistic well-posedness argument in a unified manner for $k\geq 4$ by adapting the method recently introduced by Deng-Nahmod-Yue in \cite{DNY19}, where they handled the nonlinear Schr\"odinger equations in dimension two with an arbitrarily high power. 
However, if we have deterministic well-posedness in the support of the Gibbs measure, we can obtain a better approximation property for the solutions, when compared to probabilistic methods. In fact, by establishing continuity of the solution map, any approximating sequence of initial data leads to a good approximating sequence for solutions of gKdV. The probabilistic methods above are usually restricted to particular approximations of initial data, such as those obtained by truncation to low frequencies.
In this paper, we shall take an approach based on deterministic well-posedness 
as in \cite{BO94} for the case $k=3$, namely we 
establish deterministic local well-posedness of $\mathcal{G}$-gKdV \eqref{gauged} in the Fourier-Lebesgue spaces containing the support of the Gibbs measure $\mu$, i.e., $\FL^{s,p}(\T)$ with $(s-1)p < -1$. 
It is worthwhile to observe that the counterexample to analytical well-posedness given in \cite{CKSTT04} also applies to the Cauchy problem in the Fourier-Lebesgue spaces $\FL^{s,p}(\T)$ for $s<\frac12$ and any $p$. Note that for $p=2$, the regularity criterion $(s-1)p<-1$ implies $s<\frac12$ for which a contraction argument does not work.
However, by choosing $p>2$, it is possible to take regularity $s$ satisfying both $(s-1)p<-1$ and $s>\frac12$. Such a choice guarantees that $\FL^{s,p}(\T)$ contains the support of the Gibbs measure while avoiding the counterexample in \cite{CKSTT04}.\footnote{This idea of bypassing some ill-posedness results in $L^2$-based Sobolev spaces by instead considering alternative spaces that contain the support of the Gibbs measure can be found in the literature; see, e.g., \cite{O09a,O09b,O10,FOW20}.} Hence, our first result is the following local well-posedness in $\FL^{s,p}(\T)$ for some $s,p$ which meet the above requirements. Indeed, the main difficulty in our approach lies in this step.
\begin{theorem}\label{th:lwp}
	For $2 < p < \infty$ there exists $\frac{1}{2}<s_*(p)<1-\frac{1}{p}$ such that $\mathcal{G}$-gKdV \eqref{gauged} is locally well-posed in $\FL^{s,p}(\T)$ for any $s>s_*(p)$. Moreover, by inverting the gauge transformation, we also obtain local well-posedness of the gKdV equation \eqref{gkdv} in $\FL^{s,p}(\T)$.
\end{theorem}

\begin{remark}
	\rm 
	We start by clarifying our notion of local well-posedness of $\mathcal{G}$-gKdV \eqref{gauged} in $\FL^{s,p}(\T)$: for any $u_0\in\FL^{s,p}(\T)$ there exists $T=T(\|u_0\|_{\FL^{s,p}})>0$ and a unique solution $u$ in $X^{s,\frac12}_{p,2}(T) \cap X^{s,0}_{p,1}(T) \embeds C\big([-T,T]; \FL^{s,p}(\T)\big)$ (see Definition~\ref{def:xsb}) which satisfies the Duhamel formulation of \eqref{gauged}:
	\begin{align*}
	u(t) = S(t) u_0 \pm \int_0^t S(t-t') \dx \big(u^k - k \PP_0(u^{k-1})u\big) (t') \, dt', \quad t\in [-T,T], 
	\end{align*}
	where $S(t)$ denotes the linear propagator. Moreover, the data-to-solution map $\Phi$ is (locally Lipschitz) continuous. 
	 Note that $\mathcal{G}_{0,t}$ is a bijection on $C\big([-T,T];\FL^{s,p}(\T)\big)$ with inverse given by
	$$\mathcal{G}^{-1}_{0,t} (u) (t,x) = u\Big( t, x \pm k \int_0^t \PP_0 \big( u^{k-1}(t') \big) \, dt' \Big).$$
	Consequently, Theorem~\ref{th:lwp} asserts the following notion of local well-posedness for the original gKdV equation \eqref{gkdv}: for any $u_0\in\FL^{s,p}(\T)$ there exist $T=T(\|u_0\|_{\FL^{s,p}})>0$ and a unique solution $u \in \mathcal{G}^{-1}_{0,t} \big( X^{s,\frac12}_{p,2}(T) \cap X^{s, 0}_{p,1}(T)\big) \subset C\big([-T,T]; \FL^{s,p}(\T) \big)$ which satisfies the Duhamel formulation of \eqref{gkdv}.
	The data-to-solution map $\Psi$ of gKdV \eqref{gkdv} can be defined as $\Psi(t) = \mathcal{R}_t \circ \mathcal{G}^{-1}_{0,t} \circ \Phi$, where $\mathcal{R}_t$ denotes the evaluation map at time $t$.
	The map $\Psi(t)$ is defined on a neighborhood of the origin in $\FL^{s,p}(\T)$ and it is continuous, but not Lipschitz or uniformly continuous in the $\FL^{s,p}$ topology due to the properties of $\mathcal{G}^{-1}_{0,t}$. Moreover, it satisfies the group property $\Psi(t+s) = \Psi(t) \Psi(s)$ for any $t,s$. See Appendix~\ref{ap:gauge} for more details on this map.

\end{remark}

\begin{remark}\label{rm:sp}
\rm
The lower bound $s_*(p)$ in Theorem~\ref{th:lwp} is the same as that for the main multilinear estimates in Proposition~\ref{prop:nonlinear}. We see from the proof of Proposition~\ref{prop:nonlinear} that $s_*(p)$ can be chosen as
\[ s_*(p)=1-\frac{1}{p}-\frac{\min (p-2,2)}{\max (2(k-1)p,8p)}.\]
This lower bound is not likely to be sharp; indeed, the critical exponents $s,p$ suggested by the scale invariance of the equation satisfy $s=1-\frac{1}{p}-\frac{2}{k-1}$.
It may be possible to improve the range of $s$ by adapting the method of \cite{DNY21, Ch}, for instance.
However, in this paper, we do not intend to determine the optimal range of $s$ and $p$ for local well-posedness, since our focus is on constructing a global-in-time flow on the support of the Gibbs measure.
\end{remark}

We prove Theorem~\ref{th:lwp} by applying the Fourier restriction norm method with the $X^{s,b}$ spaces adapted to the Fourier-Lebesgue setting (see Definition~\ref{def:xsb}). The method reduces to establishing a fundamental nonlinear estimate, where the main difficulty lies in controlling the derivative in the nonlinearity. To overcome this derivative loss, we want to exploit the multilinear dispersion by analyzing the phase function
\begin{equation*}
\phi_k(n,n_1, \ldots, n_k) = n^3 - n_1^3 - \ldots - n_k^3
\end{equation*}
on the hyperplane $n=n_1+\ldots+n_k$. For KdV ($k=2$) and mKdV ($k=3$), the corresponding phase functions $\phi_2$ and $\phi_3$ are known to factorize, providing an explicit characterization of the resonant set, where $\phi_k(n,n_1,\ldots,n_k) =0$. Unfortunately, such factorizations are no longer available for $\phi_k$ when $k\geq 4$, complicating the study of the resonant frequency regions. In fact, the failure of analyticity of the solution map in $H^s(\T)$ and $\FL^{s,p}(\T)$ with $s<\frac12$ in \cite{CKSTT04} is due to the failure of the corresponding nonlinear estimate in the region where $|\phi_k(n,n_1, \ldots, n_k) | \ll \max(|n_1|, \ldots, |n_k|)$. Our approach is inspired by the ``bilinear+multilinear'' strategy in the work of Colliander-Keel-Staffilani-Takaoka-Tao \cite{CKSTT03, CKSTT04}. Instead of starting by showing a bilinear estimate, we first pursue a more careful description of the frequency space by comparing $\phi_k(n,n_1, \ldots,n_k)$ with the phase functions $\phi_2(n,n_1,n-n_1)$ and $\phi_3(n,n_1,n_2,n-n_1-n_2)$ associated with KdV and mKdV, respectively. 
Moreover, we further exploit the multilinear dispersion in the form of bilinear and trilinear Strichartz estimates, which are Fourier-Lebesgue analogues of the periodic $L^4$- and $L^6$-Strichartz estimates, respectively. 
The idea of multilinearizing periodic Strichartz estimates has been used in $L^2$-based Sobolev spaces; see \cite{CKSTT04} for gKdV, \cite{HerrTataruTzv2011} for the nonlinear Schr\"odinger (NLS) equation on $\T^3$ and \cite{GH} for the derivative NLS on $\T$, for example.

	Before discussing its invariance, we must guarantee that the Gibbs measure $\mu$ is a well-defined probability measure on $\FL^{s,p}(\T)$. In particular, we need the weight $e^{\mp \frac{1}{k+1}\int_\T u^{k+1} dx}$ to be integrable with respect to the Gaussian measure $\rho$ in \eqref{gaussian}. 
	 In the defocusing case, `$+$' sign in \eqref{gkdv} and odd $k \geq 3$, it follows from the Sobolev embedding that $\mu$ is a well-defined probability measure on $\FL^{s,p}(\T)$ for $1\leq p \leq \infty$ and $s \in (1 - \frac1p - \frac{1}{k+1}, 1 - \frac1p)$. 
	 However, in the non-defocusing case, `$-$' sign in \eqref{gkdv} or even $k\geq 2$, the quantity $e^{\mp \frac{1}{k+1}\int_\T u^{k+1} dx}$ is unbounded on $\FL^{s,p}(\T)$ and the measure \eqref{measure} is not normalizable. To bypass this difficulty, Lebowitz-Rose-Speer \cite{LRS} and Bourgain \cite{BO94} introduced a mass cutoff and studied the following Gibbs measure instead
	 \begin{equation}
	 	d\mu = Z^{-1} \1_{\{\|u\|_{L^2} \leq R\}} e^{-H(u)} du.\label{mu_focusing}
	 \end{equation}
	They showed that the measure $\mu$ in \eqref{mu_focusing} is only normalizable for $1\leq k \leq 5$ and an appropriate choice of $R$. The normalizability at the optimal threshold for $k=5$ was recently shown by Oh-Sosoe-Tolomeo \cite{OhSosoeTolomeo}. See Theorem~\ref{th:focusing} for more details.

Following the strategy in \cite{BO94}, we start by proving the invariance of the Gibbs measures associated with the following truncated dynamics
\begin{equation}\label{truncated}
\begin{cases}
\dt u_N + \dx^3 u_N = \pm \PP_{\leq N} \dx \big( (\PP_{\leq N} u_N)^{k} - k \PP_0\big((\PP_{\leq N} u_N)^{k-1}\big) \PP_{\leq N} u_N \big),\\
u_N\vert_{t=0} =  u_0,
\end{cases}
\end{equation}
where $\PP_{\leq N}$ denotes the Dirichlet projection onto frequencies $\{|n| \leq N\}$. 
Unfortunately, the Hamiltonian structure of \eqref{truncated} is disrupted by the gauge transformation.
Therefore, the invariance of the corresponding Gibbs measures does not follow immediately from Liouville's Theorem. A similar difficulty was found by Nahmod-Oh-Rey-Bellet-Staffilani when studying the Gibbs measure for derivative nonlinear Schr\"odinger equation in \cite{NahOhBelletSta12}. See Remark~\ref{rm:DNLS} for additional details. As a consequence, we must establish conservation of the mass and of the Hamiltonian for \eqref{truncated} as well as the invariance of the finite dimensional Lebesgue measures under the flow of \eqref{truncated}. Then, using the invariance of the finite dimensional Gibbs measures for \eqref{truncated}, we extend solutions of \eqref{gauged} globally-in-time and also establish the invariance of $\mu$ under its flow.

\begin{theorem}\label{th:invariance}
	Assume one of the following conditions:
	
	\noi {\rm(a)} defocusing case: `$+$' sign in \eqref{gkdv} and $k$ odd;

	\noi {\rm(b)} non-defocusing case: `$+$' sign in \eqref{gkdv} and $k=4$, or `$-$' sign in \eqref{gkdv} and $3 \leq k \leq 5$, with mass $0<R\leq \|Q\|_{L^2(\R)}$ if $k=5$ and $0<R<\infty$ otherwise. Here, $Q$ denotes the (unique) optimizer of the Gagliardo-Nirenberg-Sobolev inequality on $\R$ \eqref{gagliardo-nirenberg} with $\|Q\|_{L^6(\R)}^6 = 3 \|\dx Q\|_{L^2(\R)}^2$.

	\noi Then, the $\Gg$-gKdV equation \eqref{gauged} is almost surely globally well-posed with respect to the Gibbs measure $\mu$ defined by \eqref{measure} for the case (a) or \eqref{mu_focusing} for the case (b): More precisely, for $2<p<\infty$, there exists a $\mu$-measurable set $\Sigma\subset \bigcap_{s<1-\frac1p}\FL^{s,p}(\T)$ of full $\mu$-measure such that for every $u_0\in\Sigma$, the $\mathcal{G}$-gKdV equation \eqref{gauged} with initial data $u_0$ has a uniquely defined global-in-time solution $u\in \bigcap_{s<1-\frac1p}C\big(\R; \FL^{s,p}(\T)\big)$. The obtained solution map $\Phi(t):u_0\mapsto u(t)$ for $\mathcal{G}$-gKdV defined on $\Sigma$ is $\mu$-measurable and satisfies the flow property
	\begin{equation}\label{flowproperty}
	\Phi(t)\Sigma = \Sigma~ \text{for all $t\in\R$,}\quad \Phi(t+s)=\Phi (t)\Phi(s) ~\text{for all $t,s\in \R$.}
	\end{equation}
Moreover, the Gibbs measure $\mu$ is invariant under the flow of $\mathcal{G}$-gKdV \eqref{gauged} in the sense that $\mu \big( \Phi(-t)A\big) =\mu (A)$ for any $\mu$-measurable set $A\subset \Sigma$ and $t\in \R$.
\end{theorem}

By inverting the gauge transformation and exploiting the invariance of the Gibbs measure under spatial translations, we obtain our main result.
\begin{theorem}\label{th:invariance_gkdv}
	Under the assumptions of Theorem~\ref{th:invariance}, for every $u_0$ in the set $\Sigma$ of full $\mu$-measure given in Theorem~\ref{th:invariance}, the gKdV equation \eqref{gkdv} with initial data $u_0$ has a uniquely defined global-in-time solution $u\in \bigcap_{s<1-\frac1p}C \big(\R; \FL^{s,p}(\T)\big)$. Moreover, the obtained solution map $\Psi (t)$ has the same flow property as in \eqref{flowproperty}, and the Gibbs measure $\mu$ is invariant under $\Psi(t)$ in the sense that \eqref{intro_invariance} holds for any $\mu$-measurable set $A\subset \Sigma$ and $t\in \R$.
\end{theorem}

\begin{remark}\rm\label{rm:ORT}
	
	(i) Theorem~\ref{th:invariance_gkdv} extends the results of Bourgain \cite{BO94} and Richards \cite{R} on the invariance of the Gibbs measure $\mu$. Our work establishes the first result on the invariance of the Gibbs measure $\mu$ in the sense of \eqref{intro_invariance} for large values $k\geq 5$.
	
	\noi(ii) A weaker notion of invariance of $\mu$ for $k\geq 5$ was established by Oh-Richards-Thomann in \cite{ORT16}. They constructed almost sure global dynamics for gKdV \eqref{gkdv}, without uniqueness, and established invariance in the following sense: for any $t\in\R$, the law $\mathcal{L}\big(u(t)\big)$ of the random variable $u(t)$ which solves \eqref{gkdv} is given by the Gibbs measure $\mu$.
	They followed the compactness argument introduced by Burq-Thomann-Tzvetkov \cite{BTT18}, exploiting the invariance of the truncated measures to construct a tight sequence of space-time measures. Although their result can be easily extended to the Fourier-Lebesgue spaces in Theorem~\ref{th:invariance}, we do not know if our solutions coincide with those in \cite{ORT16}. Due to the lack of uniqueness of solutions in \cite{ORT16} and the conditional uniqueness of our result, we cannot directly compare these solutions.
	
	\noi (iii) In the optimal threshold for the non-defocusing case, i.e., when $k=5$ with `$-$' sign in \eqref{gkdv} and $R = \|Q\|_{L^2(\R)}$, we can still realize the Gibbs measure $\mu$ as a weighted Gaussian measure. However, we do not have the $L^q(d\rho)$-integrability of the corresponding density for $q>1$
	, which prevents us from directly applying Bourgain's invariant measure argument. As mentioned in \cite{OhSosoeTolomeo}, this difficulty can be bypassed by using the corresponding result for $R=\|Q\|_{L^2(\R)} - \dl$, $\dl>0$, and the dominated convergence theorem as $\dl\to0$. See Appendix~\ref{ap:threshold} for further details.
\end{remark}

\begin{remark}\rm\label{rm:mean}
	For simplicity, we have restricted our discussion to mean zero initial data. However, our results extend to general data in $\FL^{s,p}(\T)$ without the mean zero condition with the following modifications. Instead of the Gaussian measure $\rho$ in \eqref{gaussian}, we consider the measure $\cj{\rho}$ induced by the Ornstein-Uhlenbeck loop
	\begin{equation}
		u^\omega(x) = \sum_{n\in\Z} \frac{g_n(\omega)}{\jb{n}} e^{inx},
		\label{OU}
	\end{equation}
which has the formal density
	\begin{equation}
		d\cj{\rho} = Z_0^{-1} e^{-\frac12 \int_\T u^2 \, dx -\frac12 \int_\T (\dx u)^2 \, dx} \, du.
		\label{gaussian_new}
	\end{equation} 
	The properties of the Gaussian measure $\rho$ discussed in Section~\ref{sec:inv}, including Lemma~\ref{lm:tail}, also hold for $\conj{\rho}$; see \cite{NahOhBelletSta12}.
	We can then define the Gibbs measure $\cj{\mu}$ with $\conj{\rho}$ as the underlying Gaussian measure, which is still normalizable (see Remarks~1.2 and 4.1 in \cite{OhSosoeTolomeo}). The main difference in the argument comes from the local well-posedness of \eqref{gauged} without the mean zero condition, since the nonlinear estimates in Section~\ref{sec:lwp} hold only for mean zero functions. This is resolved by defining the solution map $\conj{\Phi}(t)$ as a concatenation of the solution maps $\Phi_\al(t)$ of \eqref{gauged} with prescribed mean $\al$. This construction is further discussed in Appendix~\ref{ap:mean}. Since the proof of the a.s. global well-posedness and invariance of the Gibbs measure without prescribing zero mean is analogous to the argument in Section~\ref{sec:inv}, we omit the details.

\end{remark}
\begin{remark}\rm\label{rm:DNLS}
	(i) In \cite{NahOhBelletSta12}, Nahmod-Oh-Rey-Bellet-Staffilani studied the derivative nonlinear Schr\"odinger equation (DNLS) on the one-dimensional torus. In particular, they constructed a weighted Wiener measure, invariant under the gauged dynamics, and established almost sure global well-posedness of DNLS in the support of said measure. 
%
%
%
	Unlike for gKdV \eqref{gkdv}, local well-posedness in the support of the measure was already available in \cite{GH}. Consequently, the main difficulty arose in the globalization process. The energy associated to the gauged dynamics was no longer conserved for truncated solutions, which required an approach reminiscent of the $I$-method to instead establish almost invariance of the truncated measures. 
	In our case, the main difficulty is in establishing the local well-posedness of $\mathcal{G}$-gKdV \eqref{gauged} in the Fourier-Lebesgue support of the measure, which was readily available for DNLS. Although we also have to prove the invariance of the finite-dimensional Lebesgue measure with respect to the truncated dynamics in \eqref{truncated}, unlike in \cite{NahOhBelletSta12}, the Hamiltonian is still conserved and we can easily show invariance of the Gibbs measures associated to~\eqref{truncated}. 
	
	\noi(ii) One additional difficulty in establishing invariance of the Gibbs measure $\mu$ under the flow of \eqref{gkdv} was due to the gauge transformation. The map $\mathcal{G}_{0,t}$ for $k\geq 4$ is only a gauge transformation when acting on space-time functions. This is a sharp contrast with DNLS, whose more involved gauge transformation is well defined as a map on $\FL^{s,p}(\T)$, allowing the authors in \cite{NahOhBelletSta12} to consider the push-forward of the measure $\mu$ by the gauge transformation. This topic was further explored for DNLS in a subsequent work \cite{NahBelletShefSta11}. In this paper, we bypass the difficulty associated with the gauge transformation by exploiting the invariance of the Gibbs measure under spatial translations.

\end{remark}
	
\smallskip
\noindent
\textbf{Organization of the paper.}
The remainder of the paper is organized as follows. In Section~\ref{sec:notation}, we introduce relevant notations, linear estimates and auxiliary results needed to establish the main nonlinear estimates. Theorem~\ref{th:lwp} is shown in Section~\ref{sec:lwp}, where we decompose the nonlinearity into non-resonant and resonant contributions, and establish corresponding nonlinear estimates. In Section~\ref{sec:inv}, we prove almost sure global well-posedness of \eqref{gauged} and the invariance of the Gibbs measure under the corresponding flow. Moreover, we establish invariance of $\mu$ under the flow of gKdV \eqref{gkdv}. Some results on properties of the gauge transformation and of the solution map are included in Appendix~\ref{ap:gauge}. The construction of the solution map without the mean zero condition is discussed in Appendix~\ref{ap:mean}, and further details on the threshold case ($k=5$ and $R=\|Q\|_{L^2(\R)}$ in Theorem~\ref{th:invariance} (b)) can be found in Appendix~\ref{ap:threshold}. Proofs of some lemmas in Section~\ref{sec:inv} are given in Appendix~\ref{ap:measure}.

\section{Notations, function spaces and linear estimates}\label{sec:notation}
We start by introducing some useful notation. For non-negative quantities $A,B$, let $A\les B$ denote an estimate of the form $A\leq CB$ for some constant $C>0$. Similarly, $A\sim B$ will denote $A\les B$ and $B\les A$, while $A\ll B$ will denote $A\leq \eps B$, for some small positive constant $\eps$.
The notations $a+$ and $a-$ represent $a+\eps$ and $a-\eps$ for arbitrarily small $\eps>0$, respectively.
Lastly, our conventions for the Fourier transform are as follows.
The Fourier transform of $u: \R\times\T \to \R$ with respect to the space variable is given by
$$\Ft_x u(t,n) = \ft{u}(t,n) =  \int_\T u(t,x) e^{-2\pi inx} \ dx.$$
The Fourier transform of $u$ with respect to the time variable is given by
$$\Ft_t u(\tau,x) = \int_\R u(t,x) e^{-2\pi it\tau} \ dt.$$
The space-time Fourier transform is denoted by $\Ft_{t,x} = \Ft_t \Ft_x$. For simplicity, we will drop the harmless factors of $2\pi$.

Let $\mathcal{S} (\R\times\T)$  denote the space of functions $u:\R\times\R\to\C$, with $u\in C^\infty(\R\times\T)$ which satisfy
\begin{align*}
u(t,x+1) = u(t,x), \quad  \sup_{(t,x) \in \R\times\T} | t^\al \partial_t^\be \partial_x^\gamma u(t,x)| < \infty, \quad \al,\be,\gamma\in\NB \cup \{ 0\} .
\end{align*}
In the following, we define the $X^{s,b}$-spaces adapted to the Fourier-Lebesgue setting (see \cite{Grock1,GH}).

\begin{definition}\label{def:xsb}
	Let $s,b\in\R$, $1\leq p,q\leq \infty$. The space $X^{s,b}_{p,q}(\R\times\T)$, abbreviated as $X^{s,b}_{p,q}$, is defined as the completion of $\mathcal{S}(\R\times\T)$ with respect to the norm\footnote{According to our definition of the Fourier transform, the weight factor $\jb{\tau -n^3}$ should be chosen as $\jb{\tau -4\pi^2n^3}$. For simplicity of the notation, we shall ignore such an inessential power of $2\pi$ in the sequel.}
	\begin{align*}
	\|u\|_{X^{s,b}_{p,q}} = \big\| \jb{n}^s \jb{\tau - n^3}^b \ft{u}(\tau, n) \big\|_{\l^p_n L^q_\tau}.
	\end{align*}
	When $p=q=2$, the $X^{s,b}_{p,q}$-spaces defined above reduce to the standard $X^{s,b}$-spaces in \cite{BO93}.
\end{definition}

Recall the following embedding. For any $1\leq p <\infty$, 
\begin{align*}
X^{s,b}_{p,q} (\T) &\embeds C (\R; \FL^{s,p}(\T)) \quad \text{for } b>\frac{1}{q'} = 1 - \frac1q.
\end{align*}
Since $X^{s, \frac12}_{p,2}$ fails to embed into $C\big(\R; \FL^{s,p}(\T) \big)$, we will consider the smaller space $Z^{s,\frac12}_{p}\embeds C\big(\R; \FL^{s,p}(\T) \big)$ defined by $Z^{s,b}_{p} := X^{s, b}_{p,2} \cap X^{s, b-\frac12}_{p,1}$.
We can also define the local-in-time version of these spaces on the interval $[-T,T]$ through the following norm
$$\| u \|_{Z^{s, b}_p (T)} =\inf \big\{ \|v\|_{Z^{s, b}_p} : \ v\vert_{[-T,T]} = u \big\},$$
where the infimum is taken over all possible extensions of $u$ on $[-T,T]$.

%
%
The following are linear estimates associated with the gKdV equation (see \cite[Lemma~7.1]{GH} for an analogous proof). Let $S(t) = e^{-t\dx^3}$ denote the linear propagator of the Airy equation. 
\begin{lemma}\label{lm:linear} Let $1\leq p <\infty$ and $s,b\in\R$. Then, the following estimates hold:
	\begin{align*}
	\big\| S(t) u_0 \big\|_{Z^{s,b}_{p}(T)}  & \les \|u_0\|_{\FL^{s,p}}, \\
	\bigg\| \int_0^t S(t-t') F(t',x) \ dt' \bigg\|_{Z^{s,\frac{1}{2}}_p (T)} & \les \|F\|_{Z^{s,-\frac{1}{2}}_{p}(T)},
	\end{align*}
	for any $0<T\leq 1$.
\end{lemma}
The following lemma allows us to gain a small power of $T$ needed to close the contraction mapping argument. It can be shown by modifying the proof for $p=2$ (see \cite[Lemma~2.11]{TAO06}).
\begin{lemma}\label{lm:time}
	Let $-\frac{1}{2} <b'\leq b <\frac{1}{2}$ and $1\leq p < \infty$. The following holds:
	\begin{equation*}
	\left\|  u \right\|_{X^{s,b'}_{p,2}(T)} \lesssim T^{b-b'} \|u\|_{X^{s,b}_{p,2}(T)},
	\end{equation*}
	for any $0<T\leq 1$.
\end{lemma}

Lastly, we include well-known results needed in the proof of the nonlinear estimate (see, e.g., \cite[Lemma 4.2]{GTV}, \cite[Lemma~5]{MWX}, and \cite[Lemma 4.1]{CKSTT04}, respectively).
\begin{lemma}\label{lm:convolution}
	Let $0\leq \al\leq \be$ such that $\al + \be>1$ and $\eps>0$. Then, we have
	\begin{align*}
	\int_\R \frac{1}{\jb{x-a}^{\al} \jb{x-b}^{\be} } dx \les \frac{1}{ \jb{a-b}^{\gamma} },
	\end{align*}
	where 
	\begin{align*}
	\quad \gamma =
	\begin{cases}
	\al + \be -1 , & \be<1, \\
	\al - \eps, & \be =1, \\
	\al, & \be>1.
	\end{cases}
	\end{align*}
\end{lemma}

\begin{lemma}\label{lm:discrete_convolution}
	Let $0\leq \al,\be<1$ such that $\al + \be > 1$.
	Then, we have 
	\begin{align*}
	\sum_{\substack{n_1,n_2\in\Z \\ n_1+n_2=n}} \frac{1}{ \jb{n_1}^{\al} \jb{n_2}^{\be} } \les \frac{1}{ \jb{n}^{\al+\be -1} },
	\end{align*}
	uniformly over $n\in\Z$.
\end{lemma}

\begin{lemma}\label{ckstt}
	If $|n_1| \geq \ldots \geq  |n_{k}|$ and $n_1 + \ldots+ n_{k} = 0$, then
	\begin{align*}
	|n^3_1 + \ldots + n_{k}^3| \les |n_1 n_2 n_3|.
	\end{align*}
\end{lemma}

Recall the phase function $\phi_k(n,n_1, \ldots, n_k) = n^3 - n_1^3 - \ldots - n_k^3$, which we will denote by $\phi$ for simplicity.
When $k=2$ (KdV) or $k=3$ (mKdV), the phase function restricted to $n=n_1+\ldots+n_k$ satisfies the following factorizations
\begin{align}
(n_1 + n_2)^3 - n_1^3 - n_2^3 & = 3(n_1+n_2) n_1 n_2, \label{factor1}\\
(n_1+n_2+n_3)^3 -n_1^3 - n_2^3 - n_3^3 & = 3(n_1+n_2)(n_1+n_3)(n_2+n_3). \label{factor2}
\end{align}
Unfortunately, analogous factorizations no longer hold for $k\geq 4$. Instead, we establish the following lemma.

\begin{lemma}\label{lm:resonance}
	Let $k \geq 4$, $n = n_1 + \ldots + n_k$ and $|n_1| \geq \ldots \geq |n_k|>0$. 
	
	\begin{enumerate}[label=\emph{\Alph*}.]
		\item If $|n| \sim |n_1| \gg |n_2|$, $n \neq n_1$, then one of the following holds
		\begin{enumerate}[label*=\emph{\arabic*.}]
			\item $|n_1|^2 |n-n_1| \les |\phi|$;
			\item $|n_1|^2 |n-n_1| \les |n_2 n_3 n_4|$.
		\end{enumerate}
		\item If $|n_1|\sim|n_2| \gg |n_3|$, $n\neq n_1$, $n\neq n_2$, $n_1+ n_2 \neq 0$, then one of the following holds
		\begin{enumerate}[label*=\emph{\arabic*.}]
			\item $|n_1|^2 |n_1+n_2| \les |\phi|$;
			\item $|n_1+n_2| \ll |n_4|$.
		\end{enumerate}
	\end{enumerate}
	
\end{lemma}
\begin{proof}
	We start by proving A. Assume that $|n_1|^2 |n-n_1| \gg \max(|\phi|, |n_2n_3n_4|)$. Using \eqref{factor1}, we can rewrite $\phi$ as follows
	$$\phi = 3 n n_1 (n-n_1) + (n-n_1)^3 - n_2^3 - \ldots - n_k^3.$$
	Since $|nn_1(n-n_1)| \sim |n_1|^2 |n-n_1|$ and using Lemma~\ref{ckstt}, we have
	$$|n_1|^2 |n-n_1| \sim |(n-n_1)^3 - n_2^3 - \ldots - n_k^3| \les |n_2n_3| \max(|n-n_1| , |n_4|). $$
	From the above estimate, we must have $|n_1|^2 |n-n_1| \les |n_2n_3n_4|$ which contradicts our initial assumption.
	To prove part B, assume that $|n_1|^2 |n_1+n_2| \gg |\phi|$ and $|n_1+n_2| \ges |n_4|$. Using \eqref{factor2}, we can rewrite $\phi$ as follows
	\begin{align*}
	\phi = 3(n-n_1) (n - n_2)(n_1 + n_2) + (n_3 + \ldots + n_k)^3 - n_3^3 - \ldots - n_k^3.
	\end{align*}
	Since $|(n-n_1)(n-n_2)(n_1+n_2) | \sim |n_1|^2 |n_1 + n_2|$, using Lemma~\ref{ckstt}, we have
	\begin{align*}
	|n_1|^2 |n_1+n_2| \sim |(n_3 + \ldots + n_k)^3 - n_3^3 - \ldots - n_k^3 | \les |n_3|^2 |n_4|. 
	\end{align*}
	From the above estimate, we must have $|n_1+n_2| \ll |n_4|$ which contradicts our assumption.
\end{proof}
	
%
%
%
\section{Nonlinear estimates and local well-posedness}\label{sec:lwp}
In this section we state and prove the main nonlinear estimates needed to show Theorem~\ref{th:lwp}, as well as proving the latter theorem through a contraction mapping argument. We will establish a nonlinear estimate for the more general multilinear operator 
\begin{equation}\label{nonlinearity}
\mathcal{N}(u_0, \ldots, u_m) = \PP(u_1 \ldots u_m)\dx u_0 - \sum_{j=1}^m \PP_0(u_j \dx u_0) \prod_{\substack{i=1\\i\neq j}}^m u_i,
\end{equation}
where $m=k-1 \geq 3$ and $\PP = \Id - \PP_0$. Note that $\pm k\mathcal{N}(u, \ldots, u)$ coincides with the nonlinearity of $\mathcal{G}$-gKdV \eqref{gauged} and that the quantities subtracted on the right-hand side of \eqref{nonlinearity} effectively remove certain resonant frequency interactions. In fact, the spatial Fourier transform of $\mathcal{N}(u_0, \ldots, u_m)$ at $n$, omitting time dependence, is given by
\begin{equation*}
 \sum_{\substack{n=n_0 + \ldots + n_m\\nn_0\cdots n_m\neq0}}\bigg(1 - \1_{\{n=n_0\}} - \sum_{j=1}^m \1_{\{n_0+n_j=0\}}  \bigg) in_0 \ft{u}_0(n_0)  \cdots \ft{u}_m(n_m),
\end{equation*}
for mean zero functions $u_0, \ldots, u_m$.

The main difficulty in estimating $\mathcal{N}(u_0, \ldots, u_m)$ lies in controlling the derivative. To that end, we want to exploit the multilinear dispersion through the phase function $\phi$ and use Lemma~\ref{lm:resonance} to guide our case separation in the nonlinearity. Due to the restrictions in Lemma~\ref{lm:resonance}, consider the following resonant regions in frequency space:
\begin{align*}
A_j(n) =
\begin{cases}
	\big\{ (n_0, \ldots, n_m)\in\Z_*^{m+1}: \ n=n_j  \big\}, & j=0, \ldots, m,\\
	\big\{ (n_0, \ldots, n_m) \in \Z_*^{m+1}: \ n_0+n_l = 0 \big\}, & -j = l = 1, \ldots, m.
\end{cases}
\end{align*}
We can decompose the nonlinearity as $\mathcal{N} = \mathcal{N}_0 + \mathcal{R}$, where the non-resonant and resonant contributions are respectively defined as 
\begin{align}
	\Ft_x \big( \mathcal{N}_0 (u_0, \ldots, u_m)\big)(n) & = \sum_{\substack{n=n_0+\ldots+n_m\\ nn_0\cdots n_m \neq 0}} \1_{\bigcap\limits_{j=-m}^m A_j^c} in_0 \ft{u}_0(n_0) \cdots \ft{u}_m(n_m), \label{nonresonant}\\
\Ft_x \big( \mathcal{R}(u_0, \ldots, u_m) \big) (n) &= \sum_{\substack{n=n_0 + \ldots + n_m \\ nn_0\cdots n_m \neq 0}}
\bigg[\sum_{J\in \mathcal{C}} (-1)^{|J|+1} \1_{\bigcap\limits_{j\in J} A_j} \bigg]in_0 \ft{u}_0(n_0) \cdots \ft{u}_m(n_m), \label{resonant}
\end{align}
where $J\in \mathcal{C}$ if $J=\{j\}$, $j=1, \ldots, m$, or $J\subset\{-m, \ldots, m\}$ and $|J|\geq 2$.

The following proposition states the main nonlinear estimates.
\begin{proposition}\label{prop:nonlinear}
	Let $u_0, \ldots, u_m$ be mean zero functions. For $2 < p < \infty$ there exists $\frac{1}{2}<s_*(p)<1-\frac{1}{p}$ such that for any $s>s_*(p)$ the following estimates hold
	\begin{align*}
	\|\mathcal{N}_0 (u_0, \ldots, u_m) \|_{Z^{s,-\frac12}_{p}(T)} & \les T^\theta \prod_{j=0}^m \|u_j\|_{Z^{s,\frac12}_p(T)},\\
	\|\mathcal{R} (u_0, \ldots, u_m) \|_{Z^{s,-\frac12}_{p}(T)} & \les T^\theta \prod_{j=0}^m \|u_j\|_{Z^{s,\frac12}_p(T)},
	\end{align*}
	for some $0<\theta <1$ and any $0<T\leq 1$.
\end{proposition}
\begin{remark}\rm
 It will suffice to show the above estimates for $v_0, \ldots, v_m$ extensions of $u_0, \ldots, u_m$ on $[-T,T]$. Consequently, in the remaining of this section we will show the estimates in $Z^{s,-\frac12}_p$ and $Z^{s,\frac12}_p$, instead of the time localized versions. Moreover, we will establish stronger estimates which allow us to gain a small power of $T$ by applying Lemma~\ref{lm:time}.
\end{remark}

Assuming Proposition~\ref{prop:nonlinear} holds, we can prove Theorem~\ref{th:lwp} for mean zero initial data. For details on the non-zero mean case, see Appendix~\ref{ap:mean}.
\begin{proof}[Proof of Theorem~\ref{th:lwp}]
	Let $(s,p)$ satisfy the assumptions in Proposition~\ref{prop:nonlinear}.
	Given $u_0 \in \FL^{s,p}(\T)$ with zero mean, define the map $\Gamma_{u_0}$ as follows
	\begin{align*}
		\Gamma_{u_0} (u)(t) := S(t) u_0 + \int_0^t S(t-t') \N(u,\ldots, u)(t') \, dt'.
	\end{align*}
	Let $R>0$ and $B_R:=\big\{u\in Z^{s,\frac12}_{p}(T) :  \ \|u\|_{Z^{s,\frac12}_{p}(T)} \leq R \big\}.$
	Using Lemma \ref{lm:linear}, Proposition \ref{prop:nonlinear}, and the fact that $\N = \N_0 + \mathcal{R}$, we have
	\begin{align}
		\norm{\Gamma_{u_0}(u) }_{Z^{s,\frac12}_{p}(T)} &\leq C_1 \|u_0\|_{\FL^{s,p}} + C_2\| \mathcal{N}(u, \ldots, u) \|_{Z^{s,-\frac12}_p(T)} \nonumber\\
		& \leq C_1 \norm{u_0}_{\FL^{s,p}} + C_3 T^\theta \|u\|^{m+1}_{Z^{s,\frac12}_p(T)} \label{lwp1}
	\end{align}
	for some $0<\theta <1$, $C_1,C_2,C_3>0$ and any $0<T\leq 1$.
	Similarly, since $\N$, $\N_0$, and $\mathcal{R}$ are multilinear maps, we have
	\begin{align}
		\norm{\Gamma_{u_0}(u) - \Gamma_{u_0}(v) }_{Z^{s,\frac12}_p(T)} &\leq C_4 T^\theta \Big( \|u\|^m_{Z^{s,\frac12}_p(T)} + \|v\|^m_{Z^{s,\frac12}_p(T)} \Big) \| u-v \|_{Z^{s,\frac12}_p(T)} \label{lwp2}
	\end{align}
	for a constant $C_4>$ and any $0<T \leq 1$.
	Choosing $R := 2 C_1 \norm{u_0}_{\FL^{s,p}}$ and $0<T=T(R)\leq 1$ such that $C_3 T^\theta R^{m} \leq \frac12$ and $C_4T^\theta R^m\leq\frac14$, it follows from \eqref{lwp1} and \eqref{lwp2} that $\Gamma_{u_0}$ is a contraction on the closed ball $B_R \subset Z^{s,\frac12}_{p} (T)$. Consequently, $\Gamma_{u_0}$ has a (unique) fixed point $u= \Gamma_{u_0}(u)$ in $B_R$, which gives a solution $u\in Z^{s,\frac12}_p(T)$ to \eqref{gauged} with initial condition $u|_{t=0}=u_0$. The uniqueness in $Z^{s,\frac12}_p(T)$ can be shown using the estimate \eqref{lwp2} with a suitably chosen $T$. 
	The local Lipschitz continuity of the data-to-solution map follows from an analogous argument.
\end{proof}

\subsection{Bilinear and trilinear Strichartz estimates}\label{sec:strichartz}

In order to show Proposition~\ref{prop:nonlinear}, we first establish bilinear and trilinear Strichartz estimates adapted to the Fourier-Lebesgue setting.
Let $\PP$ denote the projection onto mean zero functions and $\PP_0$ denote the mean. The following lemma generalizes the periodic $L^4$-Strichartz of Bourgain in \cite{BO93} to the Fourier-Lebesgue setting.
\begin{lemma}
	The following estimate holds for any $2\leq p\leq\infty$ and $b>\max \{ \frac13, \frac{3p-2}{8p}\}$
	\begin{equation}
	\| \PP \big(\PP u_1 \cdot \PP u_2 \big) \|_{X^{0,0}_{p,2}} \les \|u_1\|_{X^{0,b}_{p,2}} \|u_2\|_{X^{0,b}_{2,2}}. \label{new_l4}
	\end{equation}
\end{lemma}

\begin{proof}
	The proof is an adaptation of the standard bilinear argument for $p=2$ (see \cite[Proposition 2.13]{TAO06} for instance) to the Fourier-Lebesgue setting. Let $M_1,M_2\geq 1$ denote dyadic numbers, $\PP_{M_j}$ the projection onto space-time frequencies $\{\jb{\tau - n^3} \sim M_j\}$ and $u_{M_1}= \PP_{M_1}u_1$, $v_{M_2} = \PP_{M_2} u_2$. 
	Since
	\begin{align*}
		\big\| \PP\big( \PP u_1 \cdot \PP u_2 \big) \big\|_{X^{0,0}_{p,2}} & \leq \sum_{M_1,M_2} \big\| \PP\big( \PP u_{M_1} \cdot \PP v_{M_2} \big)  \big\|_{X^{0,0}_{p,2}},
	\end{align*}
	it suffices to show that 
	\begin{align}\label{l4_aux}
		\big\| \PP\big( \PP u_{M_1} \cdot \PP v_{M_2} \big) \big\|_{X^{0,0}_{p,2}} & \les M_1^{b} M_2^{b} \|u_{M_1}\|_{X^{0,0}_{p,2}} \|v_{M_2}\|_{X^{0,0}_{2,2}}.
	\end{align}
	for any  $b>\max \{ \frac13, \frac{3p-2}{8p}\}$. We assume $M_1\leq M_2$, while the same proof applies to the other case. 
	Using H\"older's inequality, we get
	 \begin{equation}
	 	\big\| \PP\big( \PP u_{M_1} \cdot \PP v_{M_2} \big) \big\|_{X^{0,0}_{p,2}} \les \Big\| M_1^\frac1q\big| A(\tau,n)\big| ^\frac1q \Big[ |\ft{u}_{M_1}|^{q'}*|\ft{v}_{M_2}|^{q'}\Big] ^{\frac{1}{q'}}\Big\|_{\ell ^p_nL^2_\tau}, \label{l4_aux2}
	 \end{equation}
 	for $q>1$, where 
 	$$A(\tau,n) = \big\{ n_1\in\Z: \ 0 \neq 3nn_1(n-n_1) = -\tau + n^3 + \mathcal{O} (M_2) \big\}.$$
  	Since we can rewrite the above condition as 
  	$\big(n_1 - \frac{n}{2}\big)^2 = \frac{1}{3n}(\tau - \frac{n^3}{4}) + \mathcal{O}(\frac{M_2}{3|n|}), $
	then we conclude that there are at most $\mathcal{O}\big( 1+(\frac{M_2}{|n|})^{\frac{1}{2}} \big)$ elements in $A(\tau,n)$. We first consider the case when $|n|>M_2$ and thus $|A(\tau ,n)|\les 1$. From \eqref{l4_aux2} with $q=2$ and Young's inequality, we have
	\begin{align*}
		\text{RHS of } \eqref{l4_aux2} & \les M_1^\frac12\Big\| \Big[ |\ft{u}_{M_1}|^2*|\ft{v}_{M_2}|^2\Big] ^\frac12 \Big\|_{\ell ^p_nL^2_\tau} = M_1^\frac12\big\| |\ft{u}_{M_1}|^2*|\ft{v}_{M_2}|^2\big\|_{\ell ^{\frac{p}{2}}_nL^1_\tau}^\frac12 \\
		& \leq M_1^\frac12 \big\| |\ft{u}_{M_1}|^2\big\|_{\ell ^{\frac{p}{2}}_nL^1_\tau}^\frac12 \big\| |\ft{v}_{M_2}|^2\big\|_{\ell ^1_nL^1_\tau}^{\frac12} = M_1^\frac12 \big\| \ft{u}_{M_1}\big\|_{\ell ^p_nL^2_\tau} \big\| \ft{v}_{M_2}\big\|_{\ell ^2_nL^2_\tau},
	\end{align*}
	which implies \eqref{l4_aux}, since $M_1\leq M_2$. For the case when $|n|\leq M_2$ and $|A(\tau ,n)|\les (\frac{M_2}{|n|})^{\frac{1}{2}}$, we set
	$$
	\Big(\frac1q,\frac1r\Big)=
	\begin{cases}
		\big( \frac13 + \eps,  \frac1p - \frac16 + \eps \big) & \text{for $2\leq p \leq 6$}, \\
		\big( \frac12 - \frac1p , 0 \big) &\text{for $6<p\leq \infty$},
	\end{cases}
	$$
	with sufficiently small $\eps >0$. Note that
$$\frac{1}{p}-\frac{1}{r}<\frac{1}{2q},\quad \frac{q'}{r}+1=\frac{q'}{p}+\frac{q'}{2},\quad 1<q'\leq 2\leq p\leq r\leq \infty .$$
	Applying first H\"older's inequality in $n$, then following the above computation, and using H\"older's inequality in $\tau$, we have
	\begin{align*}
		\text{RHS of } \eqref{l4_aux2} & \les M_1^\frac1q M_2^{\frac{1}{2q}} \Big\| |n|^{-\frac{1}{2q}}\Big[ |\ft{u}_{M_1}|^{q'}*|\ft{v}_{M_2}|^{q'}\Big] ^{\frac{1}{q'}}\Big\|_{\ell ^p_nL^2_\tau} \\
		& \les M_1^\frac1q M_2^{\frac{1}{2q}} \Big\| \Big[ |\ft{u}_{M_1}|^{q'}*|\ft{v}_{M_2}|^{q'}\Big] ^{\frac{1}{q'}}\Big\|_{\ell ^r_nL^2_\tau}  \\
		& \leq M_1^\frac1q M_2^{\frac{1}{2q}} \big\| \ft{u}_{M_1}\big\| _{\ell ^p_nL^{q'}_\tau} \big\| \ft{v}_{M_2}\big\|_{\ell ^2_nL^2_\tau} \\
		& \les M_1^\frac12 M_2^{\frac{1}{2q}}  \big\| \ft{u}_{M_1}\big\|_{\ell ^p_nL^2_\tau} \big\| \ft{v}_{M_2}\big\|_{\ell ^2_nL^2_\tau}. 
	\end{align*}
	Since $M_1\leq M_2$ and $\frac{1}{2}+\frac{1}{2q}\leq 2\max \{ \frac13, \frac{3p-2}{8p}\}+\frac{\eps}{2}$, we obtain \eqref{l4_aux}, from which the estimate follows.
\end{proof}

We can then establish the following estimate.
\begin{lemma}\label{lm:bilinear_new}
	The following estimate holds for any $2\leq p \leq \infty$
	\begin{equation}
		\big\| \PP\big(\PP u_1 \cdot \PP u_2 \big) \big\|_{X^{0,-\frac12+}_{2,2}} \les \|u_1\|_{X^{0,\frac12-}_{p,2}} \|u_2\|_{X^{0,0}_{p',2}} . \label{bilinear_new}
	\end{equation}
\end{lemma}
\begin{proof}
	Using duality and \eqref{new_l4} we obtain
	\begin{align*}
		\big\| \PP \big(\PP u_1 \cdot \PP u_2 \big) \big\|_{X^{0,-\frac12+}_{2,2}} & \les \sup_{\|u_3\|_{ X^{0,\frac12-}_{2,2}}\leq 1}\bigg| \intt_{0 = \tau_1 + \tau_2 + \tau_3} \sum_{\substack{0=n_1+n_2+n_3\\ n_1n_2n_3\neq0}} \ft{u}_1(\tau_1, n_1) \ft{u}_2(\tau_2, n_2) \ft{u}_3(\tau_3,n_3) \bigg| \\
		& \leq  \sup_{\|u_3\|_{ X^{0,\frac12-}_{2,2}}\leq 1} \|u_2\|_{X^{0,0}_{p',2}}\| \PP \big(\PP u_1 \cdot \PP u_3 \big) \|_{X^{0,0}_{p,2}} \\
		& \les  \sup_{\|u_3\|_{ X^{0,\frac12-}_{2,2}}\leq 1} \|u_2\|_{X^{0,0}_{p',2}} \|u_1\|_{X^{0,\frac12-}_{p,2}} \|u_3\|_{X^{0,\frac12-}_{2,2}} \\
		& \les \|u_1\|_{X^{0,\frac12-}_{p,2}} \|u_2\|_{X^{0,0}_{p',2}},
	\end{align*}
	as intended.
\end{proof}

\begin{lemma}
	The following estimate holds for any $1\leq p,q \leq \infty$
	\begin{align}
	\| \PP_0 (u_1 u_2) u_3\|_{X^{0,0}_{p,2}} & \les \|u_1\|_{X^{0,\frac13+}_{q,2}} \|u_2\|_{X^{0,\frac13+}_{q',2}} \|u_3\|_{X^{0,\frac13+}_{p,2}}. \label{trilinear0}
	\end{align}
\end{lemma}
\begin{proof}
	By Young's and H\"older's inequalities, it follows that
	\begin{align*}
	\|\PP_0(u_1 u_2) u_3\|_{X^{0,0}_{p,2}} & \les \bigg\| \sum_{n_1} \intt_{\tau = \tau_1 + \tau_2 + \tau_3} \ft{u}_1(\tau_1,n_1) \ft{u}_2(\tau_2, -n_1) \ft{u}_3(\tau_3,n) \bigg\|_{\l^p_n L^2_\tau} \\
	& \les \|u_1\|_{X^{0,0}_{q,\frac65}} \|u_2\|_{X^{0,0}_{q', \frac65}} \|u_3\|_{X^{0,0}_{p,\frac65}} \\
	& \les \|u_1\|_{X^{0,\frac13+}_{q,2}} \|u_2\|_{X^{0,\frac13+}_{q', 2}} \|u_3\|_{X^{0,\frac13+}_{p,2}}.
	\end{align*}
\end{proof}

The following trilinear estimate can be seen as a multilinear analogue of the $L^6$-Strichartz in \cite{BO93} adapted to the Fourier-Lebesgue spaces.
\begin{lemma}
	Let $2\leq p\leq \infty$, $n_j$ denote the spatial frequency corresponding to $\ft{u}_j$, $j=1,2,3$, and assume that $(n_1+n_2) (n_1+n_3)(n_2+n_3) \neq 0$. We have the following estimate
	\begin{align}
	\| u_1 u_2 u_3\|_{X^{0,0}_{p,2}} &\les \|u_1\|_{X^{0+, \frac12}_{p,2}} \|u_2\|_{X^{0+, \frac12}_{2,2}} \|u_3\|_{X^{0+,\frac12}_{2,2}}.\label{l6-p}
	\end{align}	
\end{lemma}

\begin{proof}
	Let $\phi = 3(n_1+n_2)(n_1+n_3)(n_2+n_3)$. Using Cauchy-Schwarz inequality and Lemma~\ref{lm:convolution}, we have
	\begin{align*}
	\|u_1u_2u_3\|_{X^{0,0}_{p,2}} 
	& = \bigg\| \sum_{n=n_1+n_2+n_3} \intt_{\tau = \tau_1 + \tau_2 + \tau_3} \prod_{j=1}^3 \ft{u}_j(\tau_j, n_j) \bigg\|_{\l^p_n L^2_\tau} \\
	& \les \bigg\| \sum_{n=n_1+n_2+n_3} \frac{1}{\jb{\tau - n^3 + \phi}^{\frac12(1-\eps)}} \bigg\| \prod_{j=1}^3 \jb{\s_j}^\frac12 \ft{u}_j (n_j, \tau_j) \bigg\|_{L^2_{\tau_2} L^2_{\tau_3}} \bigg\|_{\l^p_n L^2_\tau} ,
	\end{align*}
	for any $\eps>0$.
	Since 
	$\jb{x+y} \les \jb{x}\jb{y}$ for any $x,y$, we have the following for any $\theta>0$
	\begin{align*}
	\frac{1}{\jb{n_1}^{2\theta} \jb{n_2}^{2\theta} \jb{n_3}^{2\theta}} = \frac{1}{\jb{n-n_2-n_3}^{2\theta} \jb{n_2}^{2\theta} \jb{n_3}^{2\theta}} \les \frac{1}{\jb{n_2+n_3}^\theta \jb{n-n_2}^\theta \jb{n-n_3}^\theta}.
	\end{align*}
	By letting $\theta = 2\eps$ and using Cauchy-Schwarz inequality, we obtain
	\begin{equation*}
	\|u_1u_2u_3\|_{X^{0,0}_{p,2}}  \les  \sup_{\tau,n} \big(\I(\tau,n)\big)^\frac12  
	\bigg\| \prod_{j=1}^3\jb{n_j}^{4\eps} \jb{\tau_j - n_j^3}^\frac12 \ft{u}_j (n_j, \tau_j) \bigg\|_{\l^p_n L^2_\tau \l^2_{n_2} \l^2_{n_3}L^2_{\tau_2} L^2_{\tau_3}} , \label{l6-aux}
	\end{equation*}
	where 
	\begin{align*}
	\I(\tau,n) = \sum_{n_2,n_3} \frac{1}{\jb{n_2+n_3}^{4\eps} \jb{n-n_2}^{4\eps} \jb{n-n_3}^{4\eps}\jb{\tau - n^3 + \phi}^{1-\eps}} .
	\end{align*}
	The estimate follows from Minkowski's inequality and showing a uniform bound on $\I(\tau,n)$. Let $a=\tau-n^3$. Then, we can rewrite $\I(\tau,n)$ and estimate it as follows
	\begin{align*}
	\I(\tau,n) 
	& = \sum_{\substack{l_2, l_3 \neq 0\\ 2n-l_2-l_3\neq 0}} \frac{1}{\jb{2n-l_2 - l_3}^{4\eps} \jb{l_2}^{4\eps} \jb{l_3}^{4\eps} \jb{a + 3(2n-l_2 - l_3)l_2 l_3}^{1-\eps}}\\
	& \les \sum_{r\neq 0} \sum_{\substack{(l_2, l_3)\\ r=(2n-l_2-l_3)l_2 l_3}}\frac{1}{\jb{r}^{4\eps} \jb{a + 3(2n-l_2 - l_3)l_2 l_3}^{1-\eps}}\\
	&  \les \sum_{r\neq 0} \frac{1}{\jb{r}^{4\eps} \jb{a + 3r}^{1-\eps}} \ \big| \big\{ (l_2,l_3) : \ r=(2n-l_2-l_3)l_2l_3\big\} \big| .
	\end{align*}
	Now, we employ the divisor bound (see, e.g., \cite[Theorem~278]{HWbook}):
\[ \forall \delta >0,\ \exists C_\delta >0\quad \text{s.t.}\quad \big| \big\{ n \in \mathbb{N}:\text{$n$ divides $N$}\big\} \big| \leq C_\delta N^\delta \quad (\forall N\in \mathbb{N}),\]
to estimate the number of $(l_2,l_3)$'s by
	\begin{align*}
	\big| \big\{ (l_2,l_3) : r=(2n-l_2-l_3)l_2l_3\big\} \big| 
	& \leq \big| \big\{ l_2 :\text{$|l_2|$ divides $|r|$}\big\} \big| \cdot \big| \big\{ l_3 :\text{$|l_3|$ divides $|r|$}\big\} \big| \\
	& \lesssim |r|^{\varepsilon'}
	\end{align*}
for any $\eps'>0$. Choosing $\eps' \leq 2\eps$, for example, gives
	\begin{align*}
	\I(\tau,n) \les \sum_{r\neq 0} \frac{1}{\jb{r}^{2\eps} \jb{a+3r}^{1-\eps}} < \infty,
	\end{align*}
	from Lemma~\ref{lm:discrete_convolution}, and the estimate follows.
\end{proof}
\subsection{Resonant contributions}\label{sec:resonant}
We start by considering $\mathcal{R}$ in \eqref{resonant} where $J$ satisfies $\{0,j\} \subset J$, $\{-j,0\} \subset J$ or $\{-j,j\} \subset J$, for some $j=1, \ldots, m$. The intended estimate essentially follows from the stronger estimate in Lemma~\ref{lm:N1N2}.
\begin{lemma}\label{lm:N1N2}
	Let $2\leq p<\infty$ and $s>1 - \frac1p - \frac{p-2}{mp}$. Then the following estimate holds
	\begin{multline}
		\bigg\| \intt_{\substack{\tau=\tau_0 + \ldots + \tau_m}}  \sum_{n=n_2 + \ldots + n_m}  \jb{n}^{s+1}|\ft{u}_0(\tau_0,n) \ft{u}_1(\tau_1,n)|\prod_{j=2}^m|\ft{u}_j(\tau_j,n_j)| \bigg\|_{\ell ^p_nL^2_\tau} \\
		\les \|u_0\|_{X^{s,0}_{p,2}}  \prod_{\substack{j=1}}^m \|u_j\|_{X^{s,0}_{p,1}}. \label{res1}
	\end{multline}
\end{lemma}
\begin{proof}
	Assume that $|n_2| \geq \ldots \geq |n_m|$, without loss of generality. Then $|n| \les |n_2|$. Using Young's and H\"older's inequalities, we have
	\begin{align*}
		\text{LHS of } \eqref{res1}& \les \|u_0\|_{X^{s, 0}_{\infty,2}} \|u_1\|_{X^{s,0}_{\infty,1}} \bigg\| \jb{n}^{1-s} \sum_{n=n_2+ \ldots + n_m} \prod_{j=2}^m \|\ft{u}_j(n_j)\|_{L^1_\tau} \bigg\|_{\l^p_n} \\
		& \les \sup_n \bigg( \sum_{n=n_2 + \ldots +n_m} \bigg| \frac{1}{\jb{n_2}^{2s-1} \jb{n_3}^s \cdots \jb{n_m}^s}  \bigg|^{p'} \bigg)^{\frac{1}{p'}} \|u_0\|_{X^{s, 0}_{p,2}} \prod_{j=1}^m \|u_j\|_{X^{s, 0}_{p,1}}.
	\end{align*}
	The estimate follows from Lemma~\ref{lm:discrete_convolution} for $s>\max(\frac12, 1 - \frac1p - \frac{p-2}{mp}) = 1 - \frac1p - \frac{p-2}{mp}$.
\end{proof}
The following lemma establishes an estimate for $\mathcal{R}$ in \eqref{resonant} when $J\subset\{-m, \ldots, -1\}$. 

\begin{lemma}
	For $m\geq 3$, $2\leq p<\infty$ and $s>1 - \frac1p - \min\big(\frac1p, \frac{p-2}{mp}, \frac{1}{2(m-3)\1_{m>4}}\big)$, we have
	\begin{multline}\label{res_Bij}
		\bigg\| \intt_{\tau = \tau_0 + \ldots + \tau_m}\sum_{n=n_2 + \ldots +n_m}  \frac{\jb{n}^s|n_2|}{\jb{\tau-n^3}^{\frac12-}} |\ft{u}_0(\tau_0, -n_2) \ft{u}_1(\tau_1, n_2) \ft{u}_2(\tau_2, n_2) | \prod_{j=3}^m |\ft{u}_j(\tau_j,n_j)| \bigg\|_{\l^p_n L^2_\tau} \\
		\les \|u_0\|_{X^{s,0}_{p,2}}\prod_{j=1}^m \|u_j\|_{X^{s,0}_{p,1}}.
	\end{multline}
\end{lemma}
\begin{proof}
	Assume that $|n_3|\geq \ldots \geq |n_m|$, without loss of generality.
	We will consider two cases: $|n_2| \ges |n_3|$ and $|n_2| \ll |n_3|$.
	If $|n_2| \ges |n_3|$, using Young's and H\"older's inequalities gives
	\begin{align*}
		\text{LHS of } \eqref{res_Bij}& \les \sup_n \bigg( \sum_{n=n_2 + \ldots + n_m} \bigg| \frac{1 }{\jb{n_2}^{2s-1} \jb{n_3}^s \cdots \jb{n_m}^s } \bigg|^{p'} \bigg)^{\frac{1}{p'}} \|u_0\|_{X^{s, 0}_{p,2}} \prod_{j=1}^m \|u_j\|_{X^{s, 0}_{p,1}}.
	\end{align*}
	The estimate follows from Lemma~\ref{lm:discrete_convolution} if 
	\begin{equation}\label{nonres_cond}
		s> 1 - \frac1p - \frac{p-2}{mp}.
	\end{equation}
	If $|n_2| \ll |n_3|$, then $|n|\les |n_3|$. Let $v(t,x) = \Ft_x^{-1}\big(\jb{\cdot}\ft{u}_0(t,-\cdot) \ft{u}_1(t,\cdot) \ft{u}_2(t,\cdot)\big)$. If $m=3$, from Young's inequality, we have
	\begin{align*}
		\text{LHS of } \eqref{res_Bij} & \les \|v \cdot D^s u_3\|_{X^{0,-\frac12+}_{p,2}} \les \|v\|_{X^{0,0}_{1,1}} \|u_3\|_{X^{s,0}_{p,2}}.
	\end{align*}
	Note that using Young's inequality in time and H\"older's in space, we have
	\begin{equation}
		\|v\|_{X^{0,0}_{1,1}} \les \prod_{j=0}^2 \|\jb{n}^\frac13 \ft{u}_j(\tau,n)\|_{\l^3_n L^1_\tau} \les  \prod_{j=0}^2 \|u_j\|_{X^{s, 0}_{p,1}}, \label{R3_v}
	\end{equation}
	for $s\geq \frac13$, $2\leq p \leq 3$ or $s>\frac23 - \frac1p$, $p>3$, which are less restrictive than \eqref{nonres_cond}.
	If $m\geq 4$ and $n_3+n_4 \neq 0$, we use Young's inequality to obtain
	\begin{align*}
		\text{LHS of } \eqref{res_Bij}& \les \big\| v \cdot D^s u_3 \cdot u_4 \cdots u_m \big\|_{X^{0,-\frac12+}_{p,2}} 
		\les \|v\|_{X^{0,0}_{1,1}} \|\PP \big(D^s u_3 \cdot u_4\big) \|_{X^{0,0}_{p,2}} \prod_{j=5}^m \|u_j\|_{X^{0,0}_{1,1}}.
	\end{align*}
	The first factor is estimated as in \eqref{R3_v}. For the remaining factors we use \eqref{new_l4} and H\"older's inequality, and the fact that $|n_4| \geq \ldots \geq |n_m|$. The estimate follows if $s>\frac12-\frac1p$ and, for $m>4$, if $s>1 - \frac1p - \frac{1}{2(m-3)}$.
	If $n_3 + n_4 = 0$, we have the following
	\begin{align*}
		\text{LHS of } \eqref{res_Bij}  &\les \bigg\| \sum_{\substack{n=n_2 + n_5 \\+ \ldots + n_m}} \|\ft{v}(n_2)\|_{L^1_\tau} \prod_{j=5}^m \|\ft{u}_j (n_j) \|_{L^1_\tau} \times \sum_{n_3} \jb{n_3}^s \|\ft{u}_3(n_3)\|_{L^2_\tau} \|\ft{u}_4(-n_3)\|_{L^1_\tau} \bigg\|_{\l^p_n} \\
		& \les \|u_3\|_{X^{\frac{s}{2}, 0}_{2,2}} \|u_4\|_{X^{\frac{s}{2}, 0}_{2,1}} \|v\|_{X^{3s-1, 0}_{p,1}} \prod_{j=5}^m \|u_j\|_{X^{s, 0}_{p,1}} \times \sup_n J(n)^{\frac{1}{p'}}
	\end{align*}
	where from Lemma~\ref{lm:discrete_convolution} given that $s>\frac13$ and, when $m>4$, that $s>1 - \frac1p - \frac{2p-3}{(m-1)p}$, we have
	\begin{align*}
		J(n) = \sum_{n = n_2 + n_5 + \ldots + n_m} \bigg|\frac{1}{\jb{n_2}^{3s-1} \jb{n_5}^s \cdots \jb{n_m}^s} \bigg|^{p'} \les 1.
	\end{align*}
	The estimate follows from H\"older's inequality given that $s> 1 - \frac2p$.
\end{proof}

Lastly, we consider $\mathcal{R}$ restricted to $J = J_+ \cup (- J_-)$, where $J_+,J_- \subset\{1,\ldots, m\}$ are \emph{disjoint} sets and $|J_+|\geq 1$. The following lemma estimates the case when $J_+ = \{1, \ldots, m\}$.
\begin{lemma}
	The following estimate holds for any $1\leq p <\infty$ and $s\geq \frac1m$
	\begin{align*} 
		\bigg\| \intt_{\tau = \tau_0 + \ldots + \tau_m} \frac{\jb{n}^s |(m-1)n| }{\jb{\tau-n^3}^{\frac12-}} |\ft{u}_0 (\tau_0, -(m-1)n) | \prod_{j=1}^m |\ft{u}_j(\tau_j, n)| \bigg\|_{\l^p_n L^2_\tau} \les \prod_{j=0}^m \|u_j\|_{X^{s, \frac12}_{p,2} \cap X^{s, 0}_{p,1}}.
	\end{align*}
\end{lemma}
\begin{proof}
	Using Young's inequality in time and taking a supremum in $n$, we can estimate the intended quantity by placing $u_0$ in $X^{s,0}_{p,2}$ and the remaining terms in $X^{\frac1m,0}_{\infty,1}$. The estimate therefore follows for $s\geq \frac1m$.
\end{proof}

For the remaining cases, we fix $J_+$ and gather the contributions from $J_-\subset \{ 1, \ldots ,m\}\setminus J_+$. Appealing to symmetry, let $J_+=\{1, \ldots, l\}$ for some $1\leq l \leq m-1$. Then, the net contribution can be rewritten as follows
\begin{align*}
	\mathcal{R}_l (u_0, \ldots, u_m) := \,& \intt_{\tau = \tau_0 + \ldots + \tau_m} \sum_{n=n_0 + \ldots + n_m} \bigg[ \sum_{J_-\subset \{ l+1,\ldots ,m\}} (-1)^{l+|J_-|+1}\1_{\big(\bigcap\limits_{i=1}^l A_i\big) \cap \big(\bigcap\limits_{j\in J_-} A_{-j} \big)} \bigg] \\
&\phantom{XXXXXXXXXX} \phantom{XXXXXXXXXX} \times in_0 \prod_{j=0}^m \ft{u}_j(\tau_j,n_j) \\
= \,& (-1)^{l+1}\intt_{\tau = \tau_0 + \ldots + \tau_m} \sum_{n=n_0 + \ldots + n_m} \1_{\big(\bigcap\limits_{i=1}^l A_i\big) \cap \big(\bigcap\limits_{j=l+1}^m A_{-j}^c \big)} in_0 \prod_{j=0}^m \ft{u}_j(\tau_j,n_j),
\end{align*}
which is estimated in the following lemma.

\begin{lemma}
	Let $1\leq l \leq m-1$ and $m\geq3$. Then, for $2 < p < \infty$ and $s > 1 - \frac1p - \min\big(\frac{p-2}{4p}, \frac{p-2}{mp}, \frac{1}{2m}\big)$, the following holds
	\begin{align}\label{res_ab}
		\bigg\| \frac{\jb{n}^s}{\jb{\tau-n^3}^{\frac12-}} \mathcal{R}_l(u_0, \ldots, u_m) \bigg\|_{\l^p_n L^2_\tau}
		\les \prod_{j=0}^m \|u_j\|_{X^{s, \frac12}_{p,2} \cap X^{s, 0}_{p,1}}.
	\end{align}
	
\end{lemma}

\begin{proof}
	Fix $1 \leq l \leq m-1$ and assume without loss of generality that $|n_{l+1}| \geq \ldots \geq |n_m|$. 
	If $l=1$, then $0 = n_0 + n_{2} +\ldots + n_m$, $n_0 + n_j \neq 0 $ for $j=2, \ldots, m$ and $|n_0| \les |n_2|$. Taking a supremum in $n$ and using Young's inequality, we obtain
	\begin{align*}
	\text{LHS of } \eqref{res_ab} & \les \bigg\| \intt_{\tau = \tau_0 + \ldots + \tau_m} \jb{n}^s \ft{u}_1(\tau_1, n) \sum_{0=n_0 + n_2 + \ldots+n_m} |n_0| \ft{u}_0(\tau_0, n_0) \prod_{j=2}^m \ft{u}(\tau_j, n_j) \bigg\|_{\l^p_n L^2_\tau} \\
	& \les \|u_1\|_{X^{s,0}_{p,1}} \Big\| \PP (D^s u_0 \cdot D^{1-s} u_2) \cdot u_3 \cdots u_m \Big\|_{X^{0,0}_{\infty,2}} \\
	& \les \|u_1\|_{X^{s,0}_{p,1}} \big\| \PP \big(D^s u_0 \cdot D^{1-s} u_2\big) \big\|_{X^{0,0}_{p,2}} \prod_{j=3}^m \|u_j\|_{X^{0,0}_{q,1}} 
	\end{align*}
	where $(m-2) = \frac1p + \frac{m-2}{q}$. Using \eqref{new_l4} and H\"older's inequality we have
	\begin{align*}
	\big\| \PP \big(D^s u_0 \cdot D^{1-s} u_2\big) \big\|_{X^{0,0}_{p,2}} \prod_{j=3}^m \|u_j\|_{X^{0,0}_{q,1}}  & \les \|u_0\|_{X^{s, \frac12}_{p,2}} \|u_2\|_{X^{1-s, \frac12}_{2,2}} \prod_{j=3}^m \|u_j\|_{X^{0,0}_{q,1}} \\
	& \les \|u_0\|_{X^{s, \frac12}_{p,2}} \|u_2\|_{X^{\frac32-s - \frac1p +, \frac12}_{p,2}} \prod_{j=3}^m \|u_j\|_{X^{\frac1q - \frac1p+,0}_{p,1}}
	\end{align*}
	and we must impose 
	\begin{equation} \label{res_aux1}
	s>\max\Big(\frac12, 1 - \frac1p - \frac{p-2}{4p}, 1 - \frac1p - \frac{1}{2m}\Big) =  1- \frac1p - \min\Big(\frac{p-2}{4p}, \frac{1}{2m}\Big).
	\end{equation}
	Now let $1 < l < m$. Then, $(1-l)n = n_0 + n_{l+1} + \ldots + n_m$ and $n_0+n_j \neq 0$ for $j=l+1, \ldots, m$. Assuming that $|n| \sim |n_0| \gg |n_{l+1}|$ and using H\"older's and Young's inequalities, we have
	\begin{align*}
	\text{LHS of } \eqref{res_ab} & \les \bigg\| \intt_{\tau = \tau_0 + \ldots + \tau_m} \jb{n}^s \prod_{i=1}^l \ft{u}_i(\tau_i , n) \sum_{\substack{(1-l) n = n_0 \\+ n_{l+1} + \ldots + n_m} } |n_0| \ft{u}_0(\tau_0, n_0) \prod_{j=l+1}^m \ft{u}_j(\tau_j, n_j) \bigg\|_{\l^p_n L^2_\tau} \\
	& \les \prod_{i=1}^l \|u_i\|_{X^{s, 0}_{\infty ,1}} \bigg\|   \sum_{\substack{(1-l) n = n_0 \\+ n_{l+1} + \ldots + n_m} } |n_0|^{1-(l-1)s } \|\ft{u}_0 (n_0)\|_{L^2_\tau} \prod_{j=l+1}^m \|\ft{u}_j( n_j)\|_{L^1_\tau} \bigg\|_{\l^p_n } \\
	& \les \|u_0\|_{X^{s, 0}_{p,2}} \prod_{i=1}^l \|u_i\|_{X^{s, 0}_{\infty ,1}} \prod_{j=l+1}^m \|u_j\|_{X^{-\alpha , 0}_{1,1}} ,
	\end{align*}
	where $\al = \frac{ls-1}{m-l}$ and we have assumed that $\alpha >0$.
	The estimate follows from Young's inequality if $s> \max\big(\frac1l, 1- \frac1p - \frac{(l-1)(p-1) -1}{mp}\big)$, less restrictive than 
	\begin{align*}
	s>\max \Big( \frac12, 1-\frac1p -\frac{p-2}{mp}\Big) = 1-\frac1p -\frac{p-2}{mp}.
	\end{align*}
	If $|n_0| \les |n_{l+1}|$, then we proceed as in the case when $l=1$, 
	\begin{align*}
	\text{LHS of } \eqref{res_ab} & \les \prod_{i=1}^l \|u_i\|_{X^{s, 0}_{p,1}} \big\| \PP \big(D^s u_0 \cdot D^{1-s} u_{l+1} \big) \cdot u_{l+2} \cdots u_m \big\|_{X^{0,0}_{\infty, 2}} \\
	& \les \big\| \PP \big(D^s u_0 \cdot D^{1-s} u_{l+1} \big) \big\|_{X^{0,0}_{p,2}} \bigg(\prod_{i=1}^l \|u_i\|_{X^{s, 0}_{p,1}} \bigg) \bigg( \prod_{j=l+2}^m \|u_j\|_{X^{0,0}_{q,1}} \bigg),
	\end{align*}
	where $(m-l-1) = \frac1p + \frac{m-l-1}{q}$.  Using \eqref{new_l4} and H\"older's inequality, we obtain
	\begin{align*}
	\text{LHS of } \eqref{res_ab}& \les \|u_0\|_{X^{s, \frac12}_{p,2}} \|u_{l+1}\|_{X^{1-s, \frac12}_{2,2} }  \bigg(\prod_{i=1}^l \|u_i\|_{X^{s, 0}_{p,1}} \bigg) \bigg( \prod_{j=l+2}^m \|u_j\|_{X^{0,0}_{q,1}} \bigg)  \\
	& \les \|u_0\|_{X^{s, \frac12}_{p,2}} \|u_{l+1}\|_{X^{\frac32-s - \frac1p+, \frac12}_{p,2} }  \bigg(\prod_{i=1}^l \|u_i\|_{X^{s, 0}_{p,1}} \bigg) \bigg( \prod_{j=l+2}^m \|u_j\|_{X^{\frac1q - \frac1p+,0}_{p,1}} \bigg)
	\end{align*}
	and the estimate follows if $s>1 - \frac1p - \min(\frac{p-2}{4p},\frac{1}{2(m-l+1)\1_{ l+2 \leq m}} )$, which is less restrictive than \eqref{res_aux1}.
\end{proof}

\subsection{Non-resonant contributions}
In this section, we establish the estimate for the non-resonant contribution $\mathcal{N}_0$ in \eqref{nonresonant}.
%
Without loss of generality, we can assume that $|n_1| \geq \ldots \geq |n_m|$. We further split the non-resonant contribution as follows
\begin{align*}
\mathcal{N}_0 &= \mathcal{N}_1 + \mathcal{N}_3 + \ldots + \mathcal{N}_m && \text{if } m \text{ is odd}, \\
\mathcal{N}_0 &= \mathcal{N}_1 + \mathcal{N}_3 + \ldots + \mathcal{N}_{m-1} && \text{if } m \text{ is even},\nonumber
\end{align*}
where $\mathcal{N}_\al$, for odd $1\leq \al \leq m-1$, corresponds to $\mathcal{N}_0$ further restricted to the region
\begin{align*}
\Lambda_\al(n) =\big\{ (n_0, \ldots, n_m) \in \Z_*^{m+1}: \ &   |n_1| \geq \ldots \geq |n_m|, \\ &n_j + n_{j+1} = 0, \ 1\leq j\leq \al-1 \text{ odd}, \\& n_\al + n_{\al+1} \neq 0 \big\}
\end{align*}
and $\Lambda_m(n) = \big\{ (n_0, \ldots, n_m)\in \Z_*^{m+1}: \ |n_1| \geq \ldots \geq |n_m| , \ n_1 + n_2 = \ldots = n_{m-2} + n_{m-1} = 0 \big\}$.

We will start by estimating the most difficult contribution $\mathcal{N}_1$. Guided by Lemma~\ref{lm:resonance}, we will consider the following case separation:


\begin{itemize}
	\item \textbf{Case 1}: $|n|\sim|n_0|\gg|n_1|$
	\begin{itemize}
		\item \textbf{Case 1.1:} $|n_0|^2 |n - n_0| \les |\phi|$
		\item \textbf{Case 1.2:} $|n_0|^2 |n-n_0| \les |n_1n_2n_3|$
	\end{itemize}
	\item \textbf{Case 2}: $|n_0| \sim |n_1| \gg |n_2|$
	\begin{itemize}
		\item \textbf{Case 2.1:} $|n_0|^2 |n_0 + n_1| \les |\phi|$
		\item \textbf{Case 2.2:} $|n_0+n_1|\ll|n_3|$
	\end{itemize}
	\item \textbf{Case 3}: $|n_0| \sim |n_1| \sim |n_2| \gg |n_3|$
	\begin{itemize}
		\item \textbf{Case 3.1:} $|(n-n_1)(n-n_2)(n_1+n_2)| \les |\phi|$
		\item \textbf{Case 3.2:} $|(n-n_1)(n-n_2)(n_1+n_2)|\les |n_0|^2 |n_3|$
	\end{itemize}
	\item \textbf{Case 4}: $|n|\sim|n_1| \gg |n_2|$, $|n_1|\gg |n_0| \gg |n_3|$
	\begin{itemize}
		\item \textbf{Case 4.1:} $|n_1|^2 |n - n_1| \les |\phi|$
		\item \textbf{Case 4.2:} $|n_1|^2 |n-n_1| \les |n_0n_2n_3|$
	\end{itemize}
	\item \textbf{Case 5}: $|n_1| \sim |n_2| \gg |n_0| \gg |n_3|$
	\begin{itemize}
		\item \textbf{Case 5.1:} $|n_1|^2 |n_1 + n_2| \les |\phi|$
		\item \textbf{Case 5.2:} $|n_1+n_2| \ll |n_3|$
	\end{itemize}
	\item \textbf{Case 6}: $|n_0| \les |n_3|$
	
\end{itemize}
We see that this covers all the cases:
First, the frequency region $|n_0|\gg |n_3|$ is divided into two subregions $|n_0|\gtrsim |n_1|$ and $|n_0|\ll |n_1|$, which are further divided into \textbf{Cases 1,2,3} and into \textbf{Cases 4,5}, respectively.
Then, Lemma~\ref{lm:resonance} divides each of \textbf{Cases 1,2,4,5} into the two subcases mentioned above, while the division of \textbf{Case 3} is based on the fact that $\phi -3(n-n_1)(n-n_2)(n_1+n_2)=(n_0+n_3+\ldots +n_m)^3-(n_0^3+n_3^3+\ldots +n_m^3)=\mathcal{O}(|n_0|^2|n_3|)$.
We also observe that in \textbf{Case 3} we have
\begin{align*}
\max (|n-n_1|,|n-n_2|,|n_1+n_2|)\ges |n_0|.
\end{align*}

\medskip
\noi \underline{\textbf{Cases 1.1--5.1:}}\\
Let $\s = \tau-n^3$ and $\s_j = \tau_j - n_j^3$, $j=0, \ldots, m$, denote the modulations. Then, we have the following upper bound for the resonance relation
$$|\phi| = | \s - \s_0 - \ldots - \s_m| \les \max(|\s|, |\s_0| , \ldots, |\s_m|) = \s_{\max},$$
which we can use to gain a power of $\phi$.
First, let $\s_{\max} = |\s|$. Using H\"older's inequality and Lemma~\ref{lm:convolution}, we have
\begin{align}
 \| \N_1(u_0, \ldots, u_m) \|_{X^{s, -1}_{p,1}} &\les \bigg\| \sum_{n=n_0 + \ldots+ n_m} \frac{\jb{n}^s |n_0|}{\jb{\phi}^\frac12}  \bigg(\intt_{\tau, \tau=\tau_0 + \ldots + \tau_m} \frac{1}{\jb{\s}\jb{\s_0}^{1-} \cdots \jb{\s_m}^{1-}} \bigg)^\frac12 \nonumber \\
 & \phantom{XXXXXXXXXXXXXXXXXX} \times  \prod_{j=0}^m \| \jb{\s_j}^{\frac12-}\ft{u}_j(n_j)\|_{L^2_\tau}  \bigg\|_{\l^p_n} \nonumber\\
 & \les \bigg\| \sum_{\substack{n=n_0 +\ldots + n_m}} \frac{\jb{n}^s |n_0| }{\jb{\phi}^{\frac12}} \prod_{j=0}^m \| \jb{\s_j}^{\frac12-} \ft{u}_j(n_j) \|_{L^2_\tau} \bigg\|_{\l^p_n}. \label{case11_aux}
\end{align}
For the $X^{s,-\frac12}_{p,2}$-norm, the same approach holds. If $\s_{\max} = |\s_j|$ for $j=0, \ldots, m$, we need to use $\frac12$ power of $\jb{\s_j}$ and proceed by duality in time and estimate the stronger norm $X^{s,-\frac12+}_{p,2}$. It then suffices to estimate \eqref{case11_aux}.
By using H\"older's inequality, we obtain
\begin{align*}
\eqref{case11_aux} 
& \les \sup_n \Big( \I_{\phi}(n)\Big)^{\frac{1}{p'}} \prod_{j=0}^m \|u_j\|_{X^{s, \frac12}_{p,2}} , 
\end{align*}
where
$$\I_{\phi}(n) =  \sum_{n=n_0 + \ldots + n_m} \bigg| \frac{\jb{n}^s |n_0| }{\jb{\phi}^{\frac12} \jb{n_0}^s \jb{n_1}^s \cdots \jb{n_m}^s } \bigg|^{p'},$$
and it suffices to bound $\I_\phi(n)$ uniformly in $n$.
To this end, we must consider the lower bound for $\phi$.
In \textbf{Case 1.1}, we have 
\begin{equation*}
\I_\phi(n) \les \sum_{n_1, \ldots, n_m} \frac{1}{\jb{n_1+\ldots+n_m}^{\frac{p'}{2}} \jb{n_1}^{sp'} \cdots \jb{n_m}^{sp'}} \les 1
\end{equation*}
by applying Lemma~\ref{lm:discrete_convolution} given that $s> 1 - \frac1p - \frac{1}{2m}$.
In \textbf{Case 2.1}, $|n_0| \sim |n_1| \gg |n_2|$, if $|n| \les |n_0+n_1|$, then
\begin{align*}
\I_\phi(n) \les \sum_{n_1, \ldots, n_m} \frac{1}{\jb{n_1}^{(s+\frac12)p'} \jb{n_2}^{sp'} \cdots \jb{n_m}^{sp'}} \les \Big(\sum_{n} \frac{1}{\jb{n}^{(s+\frac{1}{2m})p'}} \Big)^m \les 1, 
\end{align*}
under the following assumption
\begin{equation}
s>\max\Big(\frac12, 1-\frac1p - \frac{1}{2m}\Big).
\label{cond1}
\end{equation}
If $|n|\gg |n_0+n_1|$, then $|n|\sim|n_2+\ldots + n_m| \les |n_2|$ and we have
\begin{align*}
\I_\phi(n) & \les \sum_{n = n_0 + \ldots+n_m} \frac{1}{\jb{ n_0+n_1}^{\frac{p'}{2}} \jb{n_0}^{2sp'}   \jb{n_3}^{sp'} \cdots \jb{n_m}^{sp'}} \\
& \les \bigg( \sum_{n_0, n_1} \frac{1}{\jb{n_0+n_1}^{(s+\frac{1}{2m}) p'} \jb{n_0}^{(s+\frac{1}{2m}) p'}} \bigg) \bigg(\sum_{n} \frac{1}{\jb{n}^{(s + \frac{1}{2m})p'}} \bigg)^{m-2} \les 1,
\end{align*}
if \eqref{cond1} holds.
In \textbf{Case 3.1}, if $|n_0| |n-n_1| |n-n_2| \les |\phi|$, we use Lemma~\ref{lm:discrete_convolution} to obtain
\begin{align*}
\I_\phi(n)  &\les \sum_{n_0, n_{2},  \ldots, n_m} \frac{1}{\jb{n_0 + n_{2}  + \ldots + n_m}^\frac{p'}{2} \jb{n-n_{2}}^\frac{p'}{2} \jb{n_0}^{(2s-\frac12)p'} \jb{n_{3}}^{sp'} \cdots \jb{n_m}^{sp'}   } \\
& \les \sum_{n_0, n_{3}, \ldots, n_m} \frac{1}{\jb{n + n_0 + n_{3} + \ldots + n_m}^{p' - 1} \jb{n_0}^{\be} \jb{n_{3}}^{\be} \cdots \jb{n_m}^{\be}} \les 1,
\end{align*}
for $\be = \frac{1}{m-1} (ms-\frac12)p'$ and the estimate follows from Lemma~\ref{lm:discrete_convolution} if \eqref{cond1} holds. If $|n_0| |n_1+n_2| |n-n_1| \les |\phi|$, then
\begin{align*}
\I_\phi (n) &\les \sum_{n_1, \ldots, n_m} \frac{1}{\jb{n_1+n_2}^{\frac{p'}{2}} \jb{n-n_1}^{\frac{p'}{2}} \jb{n_2}^{(2s-\frac12)p'} \jb{n_3}^{sp'} \cdots \jb{n_m}^{sp'}   } \\
& \les \sum_{n_2, \ldots, n_m} \frac{1}{\jb{n+n_2}^{p'-1} \jb{n_2}^{(2s-\frac12)p'} \jb{n_3}^{sp'} \cdots \jb{n_m}^{sp'}} \les 1
\end{align*}
proceeding as in the previous cases by splitting the power of $\jb{n_2}$ between the other frequencies. By exchanging the roles of $n_1$ and $n_2$, we obtain the estimate when $|n_0| |n_1+n_2| |n-n_2| \les |\phi|$.
In \textbf{Case 4.1}, the estimate follows from that of Case 1.1, by exchanging the roles of $n_0$ and $n_1$. Similarly, the estimate in \textbf{Case 5.1} follows from that of Case 2.1 by exchanging the roles of $(n_0,n_1)$ with $(n_1, n_2)$.

\medskip
In Cases 1.2--5.2 and Case 6, we can no longer use the largest modulation. However, note that it suffices to control the stronger norm $X^{s, -\frac12+}_{p,2}$.

\noi \underline{\textbf{Case 1.2:}}\\
Here, we have $|n|\sim|n_0| \gg |n_1|$ and $|n_0|^2|n - n_0| \les |n_1n_2n_3|$. Thus, we can control the multiplier as follows
\begin{align*}
\jb{n}^s |n_0| \les \jb{n_0}^s \frac{|n_1n_2n_3|^\frac12}{|n_1 + \ldots + n_m|^\frac12}.
\end{align*}
Using Lemma~\ref{lm:bilinear_new}, we have
\begin{align*}
\| \mathcal{N}_1(u_0, \ldots, u_m) \|_{X^{s, -\frac12+}_{p,2}} 
& \les \big\| (D^s u_0) \cdot D^{-\frac12}\big( \PP (D^\frac12 u_1 \cdot D^\frac12 u_{2}) \cdot D^\frac12 u_{3} \cdot u_{4} \cdots u_m \big) \big\|_{X^{0,-\frac12+}_{p,2}} \\
& \les \|u_0\|_{X^{s, \frac12}_{p,2}} \big\| \PP (D^\frac12 u_1 \cdot D^\frac12 u_{2}) \cdot D^\frac12 u_{3} \cdot  u_{4} \cdots u_m \big\|_{X^{-\frac12, 0}_{p',2}},
\intertext{and then using H\"older's inequality in $n$ (and also Young's inequality if $m \geq 4$),}
& \les \|u_0\|_{X^{s, \frac12}_{p,2}}  \big\| \PP (D^{\frac12} u_1 \cdot D^{\frac12} u_{2}) \cdot D^{\frac12} u_{3} \cdot u_{4} \cdots u_m \big\|_{X^{0, 0}_{q,2}} \\
& \les \|u_0\|_{X^{s, \frac12}_{p,2}} \big\| \PP(D^{\frac12} u_1 \cdot D^{\frac12} u_{2}) \cdot D^{\frac12} u_{3} \big\|_{X^{0,0}_{q,2}} \prod_{j=4}^m \|u_j\|_{X^{0,0}_{1,1}} ,
\end{align*}
where $q=p$ for $2< p <4$ and $q=\frac{2p}{p-2}-$ for $4\leq p<\infty$. Applying \eqref{trilinear0} or \eqref{l6-p}, we obtain
\begin{multline*}
\big\| D^\frac12 u_{2} \cdot \PP_0 (D^\frac12 u_1 \cdot D^\frac12 u_{3}) \big\|_{X^{0,0}_{q,2}}  + \big\| D^\frac12 u_1 \cdot \PP_0 (D^\frac12 u_{2} \cdot D^\frac12 u_{3}) \big\|_{X^{0,0}_{q,2}}\\
+ \big\| D^\frac12 u_1 \cdot D^\frac12 u_{2} \cdot D^\frac12 u_{3} \1_{\phi' \neq 0} \big\|_{X^{0,0}_{q,2}} \\
\les \|u_1\|_{X^{\frac12, \frac12}_{q,2} } \|u_{2}\|_{X^{\frac12, \frac12}_{q,2}} \|u_{3}\|_{X^{\frac12, \frac12}_{q',2}} + \|u_1\|_{X^{\frac12+, \frac12}_{q,2} } \|u_{2}\|_{X^{\frac12+, \frac12}_{2,2}} \|u_{3}\|_{X^{\frac12+, \frac12}_{2,2}},
\end{multline*}
where $\phi' =(n_1 + n_{2}) (n_1 + n_{3}) (n_{2} + n_{3})$.
Using the fact that $|n_1| \geq \ldots \geq |n_m|$, the estimate follows from H\"older's inequality if $2< p<4$ and $s >  1 - \frac1p - \frac{p-2}{2pm}$, or $4\leq p<\infty$  and $s> 1 - \frac1p - \frac{1}{pm}$.

\medskip 

\noi\underline{\textbf{Case 2.2:}}\\
In this case we have $|n_0|\sim|n_1|\gg|n_2|$ and $|n_0+n_1| \ll |n_3|$. Thus, $|n| \les |n_2|$, $|n_0 + n_1 + n_3 + \ldots + n_m| \les |n_3|$ and we can estimate the multiplier as 
\begin{align*}
\jb{n}^s |n_0| \les \jb{n_{2}}^s \jb{n_0}^\frac12 \jb{n_1}^\frac12 \les \frac{\jb{n_{2}}^s \jb{n_0}^\frac12 \jb{n_1}^\frac12 \jb{n_3}^\frac12}{\jb{n_0 + n_1 + n_{3} + \ldots + n_m}^\frac12}.
\end{align*}
Using \eqref{bilinear_new} since $n n_{2}(n-n_{2} ) \neq 0$ and Young's inequality, we have
\begin{align*}
\|\mathcal{N}_1 (u_0, \ldots, u_m) \|_{X^{s, -\frac12+}_{p,2}} 
& \les \big\| (D^{s} u_{2}) D^{-\frac12} \big( D^{\frac12} u_0 \cdot D^\frac12 u_1 \cdot D^\frac12 u_{3} \cdot u_{4} \cdots u_m \big) \big\|_{X^{0,-\frac12+}_{p,2}} \\
& \les \|u_{2}\|_{X^{s, \frac12}_{p,2}} \big\| D^{\frac12} u_0 \cdot D^\frac12 u_1 \cdot D^\frac12 u_{3} \cdot u_{4} \cdots u_m \big\|_{X^{-\frac12,0}_{p',2}}.
\end{align*}
The intended estimate follows from the argument in Case 1.2 exchanging the roles of $(u_0, u_1, u_2,u_3)$ by $(u_2, u_0, u_1, u_3)$. Note that $(n_0+n_1)(n_0 + n_{3} ) (n_1 + n_{3}) \neq 0$, with the last factor nonzero because $|n_1| \gg |n_{3}|$.

\medskip 

\noi\underline{\textbf{Case 3.2:}}\\
Since $|n_0| \sim |n_1| \sim |n_{2}| \gg |n_{3}|$ and
\begin{equation*}
|(n-n_1)(n-n_{2})(n_1+n_{2})| \les |n_0|^2 |n_{3}|, \label{aux224}
\end{equation*}
for $N_{\min} = \min(|n - n_1|, |n-n_{2}|, |n_1 + n_{2}|)$, we get
$$N_{\min}^2|n_0| \les |n_0|^2 |n_{3}| \implies N_{\min} \les |n_0n_{3}|^\frac12.$$
If $N_{\min} = |n_1+ n_{2}|$, then $|n_1 + \ldots + n_m| \les |n_1 + n_{2}| + |n_{3}| \les |n_0 n_{3}|^\frac12$ and we can estimate the multiplier as follows
\begin{align*}
\jb{n}^s |n_0| & \les \jb{n_0}^s |n_1 n_{2}|^\frac12 \frac{|n_0 n_{3}|^{\frac{\dl}{2}}}{|n_1 + \ldots + n_m|^{\dl}}\sim \jb{n_0}^s \frac{|n_1|^{\frac12+\frac{\dl}{4}}|n_2|^{\frac12+\frac{\dl}{4}}|n_3|^{\frac{\dl}{2}}}{|n_1 + \ldots + n_m|^{\dl}},
\end{align*}
where $\dl = \frac32 - \frac3p +$ for $2< p< 4$ and $\dl = 1 - \frac1p +$ for $4\leq p<\infty$.
Using \eqref{bilinear_new} and Young's inequality, we obtain the following
\begin{align*}
\|\mathcal{N}_1 (u_0, \ldots, u_m) \|_{X^{s, -\frac12+}_{p,2}} 
& \les \bigg\| D^s u_0 \cdot D^{-\dl} \big( D^{\frac12+ \frac{\dl}{4}} u_1 \cdot D^{\frac12+ \frac{\dl}{4}} u_{2} \cdot D^{\frac{\dl}{2}} u_{3} \cdot u_{4} \cdots u_m \big) \bigg\|_{X^{0,-\frac12+}_{p,2}} \\
& \les \|u_0\|_{X^{s, \frac12}_{p,2}} \big\| D^{\frac12+ \frac{\dl}{4}} u_1 \cdot D^{\frac12+ \frac{\dl}{4}} u_{2} \cdot D^{\frac{\dl}{2}} u_{3} \cdot u_{4} \cdots u_m \big\|_{X^{-\dl, 0}_{p',2}}\\
& \les \|u_0\|_{X^{s, \frac12}_{p,2}} \big\| D^{\frac12+ \frac{\dl}{4}} u_1 \cdot D^{\frac12+ \frac{\dl}{4}} u_{2} \cdot D^{\frac{\dl}{2}} u_{3} \big\|_{X^{0,0}_{q,2}} \prod_{j=4}^m \|u_j\|_{X^{0,0}_{1,1}},
\end{align*}
where $q=\frac{2p}{4-p}$ for $2<p<4$ and $q=\infty$ for $4\leq p<\infty$, which satisfies $\frac{1}{p'}-\frac1q <\dl$ and $2<q\leq \infty$.
Since $n_1 + n_{2} \neq0$ and $|n_1|\sim|n_{2}| \gg|n_{3}|$, then $(n_1+n_{2}) (n_1 + n_{3}) (n_{2} + n_{3}) \neq 0$ and we can apply \eqref{l6-p} to obtain
\begin{align*}
&\big\| D^{\frac12+ \frac{\dl}{4}} u_1 \cdot D^{\frac12+ \frac{\dl}{4}} u_{2} \cdot D^{\frac{\dl}{2}} u_{3} \big\|_{X^{0,0}_{q,2}} \\
& \les \min \big(\|u_1\|_{X^{\frac12+\frac{\dl}{4}+, \frac12}_{q,2}} \|u_{2}\|_{X^{\frac12+ \frac{\dl}{4}+, \frac12}_{2,2}} ,\|u_1\|_{X^{\frac12+\frac{\dl}{4}+, \frac12}_{2,2}} \|u_{2}\|_{X^{\frac12+ \frac{\dl}{4}+, \frac12}_{q,2}} \big)  \|u_{3} \|_{X^{\frac{\dl}{2}+, \frac12}_{2,2}} \\
& \les \|u_1\|_{X^{\frac12+\frac{\dl}{4}+, \frac12}_{r,2}} \|u_{2}\|_{X^{\frac12+\frac{\dl}{4}+, \frac12}_{r,2}}  \|u_{3} \|_{X^{\frac{\dl}{2}+, \frac12}_{2,2}} ,
\end{align*} 
using multilinear interpolation for the last inequality, where $r=p$ for $2<p<4$ and $r=4$ for $4\leq p<\infty$.
The estimate follows if $2< p < 4$ and
\begin{align*}
s& >\max\Big( 1 - \frac1p - \frac{p-2}{8p}, 1 - \frac1p - \frac{1}{3p}, \Big( 1 - \frac1p - \frac{1}{mp}\Big) \1_{m\geq 4 } \Big) \\
&= 1 - \frac1p - \min\Big( \frac{p-2}{8p}, \frac{1}{mp}\Big),
\end{align*}
or $4\leq p<\infty$ and
\begin{align*}
s& >\max\Big(1 - \frac1p - \frac{1}{4p}, 1 - \frac1p - \frac{1}{3p}, \Big( 1 - \frac1p - \frac{1}{mp}\Big) \1_{m\geq 4 } \Big) \\
&= 1 - \frac1p - \min\Big( \frac{1}{4p}, \frac{1}{mp}\Big).
\end{align*}
If $N_{\min} = |n-n_1| = | n_0 + n_{2} + \ldots+ n_m|$, then
\begin{align*}
\jb{n}^s |n_0| & \les \jb{n_1}^s |n_0 n_{2}|^\frac12 \frac{|n_0 n_{3}|^{\frac{\dl}{2}}}{|n_0 + n_2 + \ldots + n_m|^{\dl}}\sim \jb{n_1}^s \frac{|n_0|^{\frac12+\frac{\dl}{4}}|n_2|^{\frac12+\frac{\dl}{4}}|n_3|^{\frac{\dl}{2}}}{|n_0 + n_2 + \ldots + n_m|^{\dl}},
\end{align*}
and the estimate follows from the previous argument, exchanging the roles of $u_0$ and $u_1$. Similarly, if $N_{\min} = |n - n_{2}| = |n_0 + n_1 + n_{3} + \ldots + n_m|$, we can control the multiplier as follows
\begin{align*}
\jb{n}^s |n_0| & \les \jb{n_2}^s |n_0 n_1|^\frac12 \frac{|n_0 n_{3}|^{\frac{\dl}{2}}}{|n_0 + n_1 + n_3 + \ldots + n_m|^{\dl}}\sim \jb{n_2}^s \frac{|n_0|^{\frac12+\frac{\dl}{4}}|n_1|^{\frac12+\frac{\dl}{4}}|n_3|^{\frac{\dl}{2}}}{|n_0 + n_1 + n_3 + \ldots + n_m|^{\dl}},
\end{align*}
and the estimate follows from the same arguments.

\medskip 

\noi\underline{\textbf{Case 4.2:}}\\
Since $|n|\sim|n_1| \gg |n_2|$, $|n_1|\gg |n_0| \gg |n_3|$ and $|n_1|^2 |n-n_1| \les |n_0n_2n_3|$, we estimate the multiplier as follows
$$\jb{n}^s |n_0| \les \jb{n_1}^s |n_0n_{2}|^\frac12 \frac{|n_{3}|^\frac12}{|n_0 + n_{2} + \ldots + n_m|^\frac12}.$$
The estimate follows from the strategy in Case 1.2, exchanging the roles of $u_0$ and $u_1$, considering two possibilities $(n_0+n_2)(n_0+n_3)(n_2+n_3) \neq 0$ and $n_2+n_3=0$, $|n_0|\gg |n_2|$.

\medskip 

\noi\underline{\textbf{Case 5.2:}}\\
In this case we have $|n_1|\sim|n_2| \gg |n_0| \gg |n_3|$ and $|n_1+n_2| \ll |n_3|$. Then, $|n| \sim |n_0|$, $|n_1 + \ldots + n_m| \les |n_3|$ and we estimate the multiplier as follows
\begin{align*}
\jb{n}^s |n_0| \les \jb{n_0}^s \frac{|n_1n_2n_3|^\frac12}{|n_1 + \ldots +n_m|^\frac12}.
\end{align*}
The estimate follows from the approach in Case 1.2, since $(n_1+n_2)(n_1+n_3)(n_2+n_3)\neq 0$.

\medskip 

\noi\underline{\textbf{Case 6:}}\\
Let $|n_0| \les |n_{3}|$. Let us first consider the case when $n_{2} + n_{3} = 0$. Using Young's and H\"older's inequalities, we obtain the following
\begin{align*}
\|\mathcal{N}_1 (u_0. \ldots, u_m) \|_{X^{s, -\frac12+}_{p,2}} 
& \les \bigg\| D^{s+ \frac1p - \frac12-}u_0 \cdot D^s u_1 \cdot \PP_0( D^{\frac32-\frac1p-s+} u_{2} \cdot u_{3} ) \cdot u_{4} \cdots u_m \bigg\|_{X^{0,-\frac12+}_{p,2}} \\
& \les \big\| D^{s+ \frac1p - \frac12-}u_0  \cdot D^s u_1 \big\|_{X^{0,0}_{p,2}} \|u_{2}\|_{X^{s, 0}_{p,1}} \|u_{3}\|_{X^{\frac52-\frac3p - 2s+, 0}_{p,1}} \prod_{j=4}^m \|u_j\|_{X^{0,0}_{1,1}}.
\end{align*}
Using \eqref{new_l4} and H\"older's inequality, we have
\begin{align*}
\big\| D^{s+ \frac1p - \frac12-}u_0  \cdot D^s u_1 \big\|_{X^{0,0}_{p,2}} & \les \|u_0\|_{X^{s+ \frac1p - \frac12-, \frac12}_{2,2}} \|u_1\|_{X^{s, \frac12}_{p,2}} 
\les \|u_0\|_{X^{s, \frac12}_{p,2}} \|u_1\|_{X^{s, \frac12}_{p,2}}.
\end{align*}
Then, the estimate follows from $|n_{3}| \geq \ldots \geq |n_m|$ given that $s>1 - \frac1p - \frac{1}{2m}$.
If $n_{2} + n_{3} \neq 0$, note that $|n_0 + n_{2} + \ldots + n_m| \les |n_{2}|$, so we can estimate the multiplier as follows
$$\jb{n}^s |n_0| \les \jb{n_1}^s |n_0n_{3}|^\frac12 \les \jb{n_1}^s \frac{|n_0n_{2}n_{3}|^\frac12}{|n_0 + n_{2} + \ldots + n_m|^\frac12}.$$
Then, we can proceed as in Case 1.2, exchanging the roles of $(u_0,u_1,u_2,u_3)$ by $(u_1,u_2,u_3,u_0)$, and using the fact that $(n_0+n_{2})(n_0+n_{3}) (n_{2} + n_{3}) \neq 0$.

This completes the estimate of $\mathcal{N}_1$.

\medskip

Lastly, we want to estimate $\mathcal{N}_\al$ for odd $3\leq \al \leq m$. Note that 
$$\mathcal{N}_\al (u_0,\ldots, u_m) = \PP_0 (u_1 u_2) \cdots \PP_0(u_{\al-2} u_{\al-1}) {\mathcal{N}}'_\al(u_0, u_\al, \ldots, u_m),$$
where
$$\Ft_x \big({\mathcal{N}}'_\al(u_0, \ldots, u_m) \big)(t,n) = \sum_{\substack{n=n_0+n_\al + \ldots+ n_m\\ nn_0\cdots n_m\neq 0\\n_{\al} + n_{\al+1} \neq 0}} in_0 \ft{u}_0(n_0) \ft{u}_\al (n_\al) \cdots \ft{u}_m (n_m).$$
For $\al=m$ or $\al = m-1$, the resonance relation satisfies $|\phi| \sim |nn_0n_m|$ and $|\phi| \sim |(n-n_0)(n-n_{m-1})(n-n_m)|$, respectively, thus we can proceed as in Cases 1.1--5.1, following the same strategy in time and using Cauchy-Schwarz inequality in space on the terms $\PP_0(u_1u_2), \ldots, \PP_0(u_{\al-2}u_{\al-1})$. For $3\leq \al \leq m-2$, an analogous case separation holds by replacing $(n_1,n_2,n_3)$ by $(n_\al, n_{\al+1}, n_{\al+2})$. In Cases 1.1--5.1 we follow the strategy mentioned above. To illustrate the strategy in the remaining cases, consider Case 1.2. Following the strategy for $\mathcal{N}_1$, we have
\begin{align*}
&\| \mathcal{N}_\al(u_0, \ldots, u_m) \|_{X^{s, -\frac12+}_{p,2}} \\
& \les \bigg\| (D^s u_0) \cdot D^{-\frac12}\bigg(   \Big(\prod_{\substack{j=1\\\text{odd}}}^{\al-2}\PP_0(u_j u_{j+1}) \Big) \cdot \PP (D^\frac12 u_\al \cdot D^\frac12 u_{\al+1}) \cdot D^\frac12 u_{\al+2} \cdot \prod_{i=\al+3}^m u_{i} \bigg) \bigg\|_{X^{0,-\frac12+}_{p,2}} \\
& \les \|u_0\|_{X^{s, \frac12}_{p,2}} \bigg\| \Big(\prod_{\substack{j=1\\\text{odd}}}^{\al-2}\PP_0(u_j u_{j+1}) \Big) \cdot \PP (D^\frac12 u_\al \cdot D^\frac12 u_{\al+1}) \cdot D^\frac12 u_{\al+2} \cdot \prod_{i=\al+3}^m u_{i} \bigg\|_{X^{-\frac12, 0}_{p',2}} \\
& \les  \|u_0\|_{X^{s, \frac12}_{p,2}} \Big( \prod_{j=1}^{\al-1} \| u_j\|_{X^{\frac12-\frac1p+, 0}_{p,1}} \Big) \bigg\| \PP (D^\frac12 u_\al \cdot D^\frac12 u_{\al+1}) \cdot D^\frac12 u_{\al+2} \cdot \prod_{i=\al+3}^m u_{i} \bigg\|_{X^{-\frac12, 0}_{p',2}},
\end{align*}
using Young's and Cauchy-Schwarz inequalities in the last step. The last term can be estimated following the same approach as for $\mathcal{N}_1$.

\section{Almost sure global well-posedness and invariance of the Gibbs measure}\label{sec:inv}

In this section, we extend the solutions of Theorem~\ref{th:lwp} globally-in-time and show invariance of the Gibbs measure under the dynamics of gKdV \eqref{gkdv}, for mean zero initial data. We closely follow the argument in \cite{NahOhBelletSta12}.

Recall that $(\Omega, \mathcal{F}, \P)$ is a probability space and $\{g_n\}_{n\in \Z_*}$, $\Z_*= \Z \setminus \{0\}$, a sequence of complex-valued standard Gaussian random variables with $g_{-n} = \conj{g_n}$. 
 We can define the Gaussian measure $\rho$ as the induced probability measure under the map
  \begin{align}
 \omega \mapsto u^\omega(x) = \sum_{n \in \Z_*} \frac{g_n(\omega)}{|n|} e^{inx} \in \bigcap_{s<1-\frac1p} \FL^{s,p}(\T) \ \text{a.s.}, \label{u}
 \end{align}
 or equivalently, as $\rho = \P \circ u^{-1}$ the push-forward of the map in \eqref{u}, with the following density
 \begin{equation*}
d\rho = Z^{-1} e^{-\frac12 \int_\T(\partial_x u)^2 } du. \label{rho}
\end{equation*}
Further details on the construction of Gaussian measures in Banach spaces can be found in \cite{gross, kuo75}, for example. Before discussing the construction of the Gibbs measure $\mu$, we recall the following tail estimate for $\rho$. 
\begin{lemma}\label{lm:tail}
Let $(s,p)$ satisfy $(s-1)p<-1$ and $K>0$. Then, the following estimate holds
\begin{align*}
	\rho\big( \| u\|_{\FL^{s,p}} >K\big) \leq C e^{-cK^2},
\end{align*}		
for some constants $C,c>0$ depending only on $s$ and $p$.
\end{lemma}
This lemma follows from the fact that $\FL^{s,p}(\T)$ is an abstract Wiener space for $(s-1)p<-1$ (see \cite{BO2011, NahOhBelletSta12}) and from Fernique's theorem \cite{fernique}. 
Now, we view the Gibbs measure $\mu$ in \eqref{measure} as a weighted Gaussian measure
\begin{equation*}
d\mu =  Z^{-1} e^{\mp\frac{1}{k+1} \int_\T u^{k+1} dx} d \rho(u). \label{aux_gibbs}
\end{equation*}
In the defocusing case ($+$ in \eqref{gkdv}) and for odd $k\geq 1$, the measure $\mu$ is well-defined as a measure in $\FL^{s,p}(\T)$ for $1\leq p \leq \infty$ and $1 - \frac1p - \frac{1}{k+1} < s< 1 - \frac1p$, and it is absolutely continuous with respect to $\rho$. This follows easily from Sobolev inequality. For the non-defocusing case, Lebowitz-Rose-Speer \cite{LRS} and Bourgain \cite{BO94} proposed the introduction of a mass cutoff and instead studied the following Gibbs measure
\begin{equation*}
d\mu = Z^{-1} \1_{\{\|u\|_{L^2} \leq R\}} e^{\mp \frac{1}{k+1} \int_\T u^{k+1} dx } d\rho(u).
\end{equation*}
This new measure is known to be normalizable as stated in the following theorem.

	\begin{theorem}[\cite{LRS, BO94, OhSosoeTolomeo}] \label{th:focusing}
	Let $k\geq 2$, $R>0$, and define $F(u)$ by
	\begin{align}
		F(u) = e^{\mp \frac{1}{k+1} \int_\T u^{k+1}dx} \1_{\{\|u\|_{L^2} \leq R\}},\label{a_F}
	\end{align}
	where `$\mp$' above corresponds to `$\pm$' in the equation \eqref{gkdv}. Then, for $1 \leq q <\infty$, we have that $F(u) \in L^q(d\rho)$ if one of the following (a), (b) holds:
	
	\noi{\rm (a)} $2\leq k \leq 4$ with `$+$' sign in \eqref{a_F} when $k=3$, and any finite $R>0$;
		
	\noi{\rm (b)} $k=5$ with `$+$' sign in \eqref{a_F}, and $0<R <\|Q\|_{L^2(\R)}$, 
	
	\noi where $Q$ is the (unique) optimizer for the Gagliardo-Nirenberg-Sobolev inequality on $\R$:
	\begin{equation}
		\|u\|^6_{L^6(\R)} \leq C \|u\|^4_{L^2(\R)} \| \dx u \|^2_{L^2(\R)},
		\label{gagliardo-nirenberg}
	\end{equation}
	
	\noi with $\|Q\|_{L^6(\R)}^6 = 3\| \dx Q\|_{L^2(\R)}^2$. Moreover, if 
	
	\noi{\rm (c)} $k=5$ with `$+$' sign in \eqref{a_F}, and $R = \|Q\|_{L^2(\R)}$,
	
	\noi then we have $F(u)\in L^1(d\rho)$.
\end{theorem}

\begin{remark}\rm\label{rem:renormalize_measure}
	\noi (i) Theorem~\ref{th:focusing} was first claimed in \cite{LRS}. Unfortunately, there was a gap in the argument for (b) as stated in \cite{CFL}. In \cite{BO94}, Bourgain presented a more analytic proof of Theorem~\ref{th:focusing} (a) for any finite $R$ and (b) for small enough $R$. The result for (c) was recently proved by Oh-Sosoe-Tolomeo in \cite{OhSosoeTolomeo}. Since in this case we do not have the $L^q$-integrability of the density $F(u)$ for $q>1$, we cannot directly apply Bourgain's invariant measure argument to prove Theorem~\ref{th:invariance}. However, using a limiting argument, we can extend almost sure global well-posedness and the invariance of the Gibbs measure to this case. We will give a proof of Theorem~\ref{th:invariance} for the cases (a), (b) in this section, and describe how to treat the threshold case (c) in Appendix~\ref{ap:threshold}.
	
	\noi (ii) The gKdV equations \eqref{gkdv} and the nonlinear Schr\"odinger equation (NLS) share a Hamiltonian, and consequently they have the same associated Gibbs measure. An NLS analogue of Theorem~\ref{th:focusing} (i.e., in the complex-valued setting and for general fractional power) was also shown in \cite{OhSosoeTolomeo}. In fact, the critical threshold $R=\| Q\|_{L^2(\R)}$ for $k=5$ is related to the existence of finite time blow-up solutions with minimal mass $\| Q\|_{L^2(\R)}$ for NLS on $\T$ due to Ogawa-Tsutsumi \cite{OgTs}. Although the quintic focusing gKdV equation \eqref{gkdv} on the real line also exhibits finite time blow-up solutions with mass arbitrarily close to $\| Q\|_{L^2(\R)}$ \cite{Merle01,MartelMerle02}, it is known \cite{MartelMerle02b} that there does not exist a blow-up solution with minimal mass (satisfying a certain condition), and moreover there are no analogous results on $\T$.

	\noi (iii) Note that the assumptions in Theorem~\ref{th:invariance} (b) follow from those in Theorem~\ref{th:focusing} needed to rigorously construct the Gibbs measure $\mu$ in the non-defocusing case. In fact, the measure is not normalizable in the non-defocusing case when $k>5$ or when $k=5$ and $R> \|Q\|_{L^2(\R)}$~\cite{LRS, OhSosoeTolomeo}. 
	
	\noi (iv) In the cases (a), (b) of Theorem~\ref{th:focusing}, it follows from its proof that, for $F_{N}(u) := F(\PP_{\leq N} u)$ and any $1\leq q <\infty$, the estimate
	\begin{equation}
	\|F_N\|_{L^q(d\rho)} \leq C < \infty \label{FNR}
	\end{equation}
holds uniformly in $N$.
\end{remark}

In the rest of this section, we focus on the cases (a), (b) of Theorem~\ref{th:focusing}. For simplicity, we choose to take the `$-$' sign in the definition of $\mu$, as it will not play a role in the results. Lastly, we state the following known result on the convergence of the truncated measures $\mu_N$ defined by
\begin{align*}
d \mu_N(u) = Z_N^{-1} F_{N}(u) d\rho(u).
\end{align*}

	\begin{lemma}\label{lm:F}
	For all $1\leq q < \infty$, we have
	\begin{align*}
	F_{N}(u) \to F(u) \quad \text{in } L^q(d\rho) \text{ as } N\to\infty.
	\end{align*}
	Moreover, for all $\eps>0$, there exists $N_0\in \NB$ such that for $N\geq N_0$ and any measurable set $A\subset \FL^{s,p}(\T)$, for $1\leq p<\infty$ and $s\in(1 - \frac1p - \frac{1}{k+1}, 1 - \frac1p)$, the following holds
	$$|\mu_N (A) - \mu(A) |<\eps.$$
\end{lemma}

Consider the following truncated gauged gKdV equation ($\mathcal{G}$-gKdV$_N$)
\begin{align}\label{truncated2}
\begin{cases}
\partial_t u_N + \partial_x^3 u_N  = k \PP_{\leq N} \Big( \partial_x (\PP_{\leq N}  u_N ) \cdot \PP ( \PP_{\leq N} u_N)^{k-1} \Big),\\
u_N\vert_{t=0} = u_0.
\end{cases}
\end{align}
The local well-posedness of \eqref{truncated2} follows from the proof of Theorem~\ref{th:lwp}, with the same time of existence $\dl \sim (1+ \|u_0\|_{\FL^{s,p}})^{-\gamma}$ as the solution  $u$ of \eqref{gauged}. Moreover, as we see below, \eqref{truncated2} is globally well-posed. Note that we can decompose $u_N$ into high and low frequencies $u_N = u_{\text{low}} + u_{\text{high}}$, which solve the following equations
 \begin{align*}
 \partial_t u_\text{high}  + \partial_x^3 u_\text{high} &= 0, \\
 \partial_t u_\text{low} + \partial_x^3 u_\text{low} &= k \PP_{\leq N} \Big(  \partial_x u_\text{low} \cdot \PP ( u_\text{low})^{k-1} \Big) ,
 \end{align*}
 allowing us to discuss the two decoupled flows $\Phi_{\text{high}}$ and $\Phi_{\text{low}}$, respectively. The high frequency part evolves linearly, therefore $\Phi_{\text{high}}(t) = S(t)\PP_{>N}$. We can view the low frequency part as a finite-dimensional system of nonlinear ODEs on the Fourier coefficients of $u_N$. In fact, for $0<|n| \leq N$ and $c_n = \ft{u_N}(n)$, we want to solve the following system for $c=\{c_n\}_{0<|n|\leq N}\in \C^{2N}$ with $c_{-n} = \conj{c}_n$,
 	\begin{align}\label{finite_system}
 	\frac{d}{dt} c_n = in^3 c_n + k \sum_{\substack{n=n_0 + \ldots + n_{k-1}\\ n\neq n_0}} in_0 c_{n_0} \cdots c_{n_{k-1}} = \mathbf{N}_n(c) .
 	\end{align}
Since $\mathbf{N}=\{\mathbf{N}_n\}_{0<|n|\leq N}$ is Lipschitz, we can conclude by the Cauchy-Lipschitz theorem that the system of ODEs is locally well-posed. Furthermore, we can extend these solutions globally-in-time since the $L^2$-norm of $u_N$ is conserved:
\begin{align}
\frac{d}{dt} M(u_N)(t) 
& = 2 \int_\T u_N \Big( - \partial_x^3 u_N + k \PP_{\leq N } \big( \partial_x \PP_{\leq N} u_N \cdot \PP (\PP_{\leq N} u_N)^{k-1} \big) \Big) \, dx \nonumber\\
&= \int_\T \partial_x \Big( \partial_x u_N \Big)^2 \, dx + \frac{2k}{k+1} \int_\T\dx \big( \PP_{\leq N} u_N\big)^{k+1} \, dx \nonumber\\
& \phantom{XXXXXXXX}-k \PP_0 \big( \PP_{\leq N} u_N \big)^{k-1} \int_\T \dx \big(\PP_{\leq N} u_N\big)^2 \,dx
= 0 .\label{mass_trunc}
\end{align}
Thus, $M(u_N)(t) = M(u_0)$. In addition, the mass is also conserved for $u_\text{low}$, $M(u_\text{low})(t) = M(\PP_{\leq N} u_0)$, and the solution of \eqref{finite_system} exists globally-in-time, proving that $u_N$ extends to a global solution of \eqref{truncated2}.
We denote the flow of $\mathcal{G}$-gKdV$_N$ \eqref{truncated2} by $\Phi_N(t)$.

We now focus on proving invariance of the Gibbs measure associated with \eqref{truncated2}. We first decompose the measure $\rho = \rho_N \otimes \rho_N^\perp $, where
\begin{align*}
d \rho_N & = Z_N^{-1} e^{-\frac12 \sum_{0<|n|\leq N} |g_n|^2} \prod_{0<|n|\leq N} d g_n, \\
d\rho_N^\perp& = \tilde{Z}_N^{-1} e^{-\frac12 \sum_{|n|>N} |g_n|^2} \prod_{|n|>N} dg_n,
\end{align*}
which are probability measures in $\FL^{s,p}(\T)$ for $s< 1 - \frac1p$.
Let $\tilde{\mu}_N$ denote the finite dimensional Gibbs measure associated with density
\begin{align*}
d \tilde{\mu}_N (u ) = Z_N^{-1} F_{N} (u) d\rho_N (u).
\end{align*}
Then, $\mu_N = \tilde{\mu}_N \otimes \rho_N^\perp$ is the Gibbs measure associated with $\mathcal{G}$-gKdV$_N$ \eqref{truncated2}.

\begin{proposition}\label{inv_trunc}
	The finite-dimensional Gibbs measure $\tilde{\mu}_N$ is invariant under the flow $\Phi_{\text{low}}$. Moreover, the Gibbs measure $\mu_N$ is invariant under the flow $\Phi_N$ of $\mathcal{G}$-gKdV$_N$ \eqref{truncated2}.
\end{proposition}
\begin{proof}
	We follow the strategy in \cite{NahOhBelletSta12}.
	We start by establishing the invariance of $\tilde{\mu}_N$ under the flow of $\Phi_{\text{low}}$. The conservation of mass for $u_\text{low}$ follows from the calculation in \eqref{mass_trunc} by replacing $u_N$ by $u_\text{low} = \PP_{\leq N} u_N$. An analogous straightforward computation establishes the conservation of the Hamiltonian for $u_\text{low}$. It remains to show the invariance of the Lebesgue measure on $\C^{2N}$ with respect to the system defined in \eqref{finite_system}. We can rewrite the system as 
	\begin{align*}
		\frac{d}{dt} a_n & = \Re\big( \mathbf{N}_n(\{a_n, b_n\} )\big) , \qquad
		\frac{d}{dt} b_n  = \Im \big( \mathbf{N}_n(\{a_n, b_n\} )\big), 
	\end{align*}
	where $c_n = a_n + i b_n$. Thus, invariance of the Lebesgue measure follows from Liouville's theorem once we establish that the divergence of the vector field vanishes:
	\begin{equation}\label{liouville}
		\sum_{1\leq |n| \leq N} \Big(\frac{\partial \Re( \mathbf{N}_n) }{\partial a_n} + \frac{\partial \Im(\mathbf{N}_n)}{ \partial b_n } \Big)= 0.
	\end{equation}
	For $1\leq |n| \leq N$, we have
	\begin{align*}
		\frac{\partial \Re( \mathbf{N}_n)}{\partial a_n} & = \frac{\partial}{\partial a_n} \bigg( -n^3b_n+\frac{k}{2} \sum_{\substack{n=n_0+\ldots +n_{k-1}\\ n\neq n_0}} \big( in_0 c_{n_0} \cdots c_{n_{k-1}} -in_0\conj{c_{n_0}}\cdots \conj{c_{n_{k-1}}}\big) \bigg) \\
		& = \frac{k}{2} \sum_{\substack{n=n_0 + \ldots + n_{k-1} \\ n\neq n_0}} \sum_{j=1}^{k-1} \bigg( in_0c_{n_0} \dl(n-n_j) \prod_{\substack{i=1\\i\neq j}}^{k-1} c_{n_i} - in_0 \conj{c_{n_0}} \dl(n-n_j) \prod_{\substack{i=1\\i\neq j}}^{k-1} \conj{c_{n_i}}  \bigg)\\
		& = \frac{k(k-1)}{2} \sum_{\substack{0=n_0+\ldots +n_{k-2}\\ n\neq n_0}} \big( in_0 c_{n_0} \cdots c_{n_{k-2}} - in_0\conj{c_{n_0}}\cdots \conj{c_{n_{k-2}}}\big) \\
		& = k(k-1)\Re \bigg( \sum_{\substack{0=n_0+\ldots +n_{k-2}\\ n\neq n_0}} in_0 c_{n_0} \cdots c_{n_{k-2}}  \bigg) .
	\end{align*}
	Similarly, we have
	\begin{align*}
		\frac{\partial \Im( \mathbf{N}_n)}{ \partial b_n} & = \frac{\partial}{\partial b_n} \bigg( n^3a_n+\frac{k}{2i} \sum_{\substack{n=n_0+\ldots +n_{k-1}\\ n\neq n_0}} \big( in_0 c_{n_0} \cdots c_{n_k} +in_0\conj{c_{n_0}}\cdots \conj{c_{n_{k-1}}}\big) \bigg) \\
		& = k(k-1)\Re \bigg( \sum_{\substack{0=n_0+\ldots +n_{k-2}\\ n\neq n_0}} in_0 c_{n_0} \cdots c_{n_{k-2}}  \bigg) .
	\end{align*}
	Since 
	\begin{align*}
		\sum_{0=n_0+\ldots+n_{k-2}} in_0 c_{n_0} \cdots c_{n_{k-2}} 
		& = \int_\T \dx u_\text{low} \cdot u_\text{low}^{k-2} \,dx = 0,
	\end{align*}
	we conclude \eqref{liouville} after summing up over $n$.
	Lastly, the invariance of $\mu_N = \tilde{\mu}_N \otimes \rho_N^\perp$ under the flow $\Phi_N = \big(\Phi_\text{low}, \Phi_\text{high}\big)$ follows from that of $\tilde{\mu}_N$ under the flow $\Phi_\text{low}$ and the invariance of Gaussian measures under rotation.
\end{proof}

Let $2 < p <\infty$ and $s_* = s_*(p)$ given by Theorem~\ref{th:lwp} such that \eqref{gkdv} and \eqref{gauged} are locally well-posed in $\FL^{s,p}(\T)$ for $s_*<s<1 - \frac1p$.
The following two lemmas can be shown through the method in \cite{BO94} (see also \cite{Tzv08, BurqTzv07, OhKdv09, NahOhBelletSta12}).
The proof of Lemma~\ref{lm:uniform} requires the tail estimate in Lemma~\ref{lm:tail}, Theorem~\ref{th:lwp}, Proposition~\ref{inv_trunc} and \eqref{FNR}. Lemma~\ref{lm:approx} is purely deterministic and follows from the local theory for $\mathcal{G}$-gKdV \eqref{gauged}. 
Proofs of these lemmas will be given in Appendix~\ref{ap:measure}.

\begin{lemma}\label{lm:uniform}
	Let $s_*<s<1 - \frac1p$. Then, there exists $C_0>0$ (independent of $s$) and $C_s>0$ such that: for all $N\in\NB$, $T\geq 1$, $0<\eps\leq \frac12$, $A \geq 1$, there exists $\Omega_N^s(T,\eps,A)\subset \FL^{s,p}(\T)$ 
	such that:
	
	\noi {\rm(a)} $\mu_N \big( \FL^{s,p}(\T) \setminus \Omega_N^s (T,\eps, A)  \big) < \eps$.
	
	\noi {\rm (b)} For $u_0 \in \Omega_N^s(T, \eps, A)$, the solution $u_N$ to \eqref{truncated2} satisfies
		\begin{equation*}
		\big\| u_N(t) \big\|_{\FL^{s,p}} \leq A C_0 C_s \big(\log \tfrac{T}{\eps}\big)^\frac12, \quad |t| \leq T.
		\end{equation*}
	
	\noi {\rm (c)} For $u_0\in \FL^{s,p}(\T)$, if the solution $u_N$ to \eqref{truncated2} satisfies
		\begin{equation*}
		\big\| u_N(t) \big\|_{\FL^{s,p}} \leq A C_s \big(\log \tfrac{T}{\eps}\big)^\frac12, \quad |t| \leq T,
		\end{equation*}
then $u_0\in \Omega _N^s(T,\eps ,A)$.
\end{lemma}

\begin{lemma}\label{lm:approx}
	For any $s_*<s<\s<1 - \frac1p$, $T\geq 1$, and $K\geq 1$, there exists $N_0\in\NB$ such that:
	
	\noi \rm{(a)} Let $N\geq N_0$ and $u_N \in C\big(\R; \FL^{\s,p}(\T)\big)$ be the solution of $\mathcal{G}$-gKdV$_N$ \eqref{truncated2} with initial data $u_0\in \FL^{\s,p}(\T)$. Assume that $\| u_N(t)\|_{\FL^{\s,p}} \leq K$ for $|t| \leq T$. Then, there exists a unique solution $u \in C\big([-T,T] ; \FL^{s,p}(\T)\big) \cap Z^{s,\frac12}_{p}(T)$ to $\mathcal{G}$-gKdV \eqref{gauged} with $u(0) = u_0$ satisfying
	$$ \| u(t)-\mathbf{P}_{\leq N}u_N(t)\|_{\F L^{s,p}}\leq  \big( \tfrac{N_0}{N} \big)^{\s -s } K,\qquad |t|\leq T. $$
	In particular, $\|u(t) \|_{\FL^{s,p}} \leq 2K$ for $|t| \leq T$.
	
	\noi \rm{(b)} Let $u\in C\big([-T,T] ; \FL^{\s,p}(\T) \big) \cap Z^{\s,\frac12}_{p}(T)$ be a solution of $\mathcal{G}$-gKdV \eqref{gauged} with $u(0) = u_0$ satisfying $\|u(t)\|_{\FL^{\s,p}} \leq K$ for $|t| \leq T$. Then, for any $N\geq N_0$, the solution $u_N$ of $\mathcal{G}$-gKdV$_N$ \eqref{truncated2} with initial data $u_0$ satisfies 
	$$\|u(t) - \PP_{\leq N} u_N(t) \|_{\FL^{s,p}} \leq \big( \tfrac{N_0}{N} \big)^{\s - s } K, \qquad |t| \leq T.$$
	In particular, $\|u_N(t) \|_{\FL^{s,p}} \leq 3 K$ for $|t| \leq T$.
\end{lemma}
\begin{remark}\rm
	We can choose, for example, $N_0 \sim \exp\big(\frac{C K^\gamma T}{\s -s}\big)$ with $\gamma =\frac{k-1}{\theta}$ and $\theta >0$ given in Proposition~\ref{prop:nonlinear}.
\end{remark}
Using Lemma~\ref{lm:uniform} and Lemma~\ref{lm:approx}, we establish almost a.s. global well-posedness of the $\mathcal{G}$-gKdV equation \eqref{gauged}.

	\begin{proposition}\label{prop:almostGWP}
		Let $s_*< s < 1 - \frac1p$, $T\geq 1$, and $0<\eps\leq \frac12$. For any $A\geq 1$, there exists $N_1=N_1(A) \in \NB$ such that the set $\Sigma_{T,\eps}^{s}(A) := \Omega_{N_1}^{\s} (T,\tfrac{\eps}{2}, A)$, with $\s =\frac12 (s+1-\frac1p)$, satisfies:
	
	\noi {\rm (a)} $\mu\big(\FL^{s,p}(\T) \setminus \Sigma_{T,\eps}^s(A) \big) < \eps$;
	
	\noi {\rm (b)} For $u_0\in\Sigma_{T,\eps}^s(A)$, there exists a unique corresponding solution $u\in C\big( [-T,T];\FL^{s,p}(\T)\big)\cap Z^{s,\frac12}_p(T)$ to $\mathcal{G}$-gKdV \eqref{gauged} on $[-T,T]$ such that 
	\begin{align*}
		\|u(t)\|_{\FL^{s,p}} \leq 2 \sqrt{2} A C_0 C_\s \big(\log \tfrac{T}{\eps} \big)^\frac12 , \quad |t| \leq T.
	\end{align*}

\end{proposition}

\begin{proof}
	Lemma~\ref{lm:uniform}(b) shows that for $u_0\in\Sigma^{s}_{T,\eps}(A)$ we have
	$$\|\Phi_{N_1}(t) (u_0) \|_{\FL^{\s,p}} \leq A C_0 C_\s \big(\log \tfrac{2T}{\eps}\big)^\frac12 , \quad |t| \leq T.$$
	From Lemma~\ref{lm:approx}(a), there exists a unique solution $u$ to $\mathcal{G}$-gKdV on $[-T,T]$ with $u(0) = u_0$ satisfying 
	\begin{align*}
		\| u(t) \|_{\FL^{s,p}} \leq 2 A C_0 C_\s \big(\log \tfrac{2T}{\eps}\big)^\frac12, \qquad |t| \leq T,
	\end{align*}
	provided $N_1$ is large enough. The intended estimate follows from $\log(2x) \leq 2 \log x$ for $x \geq 2$.
%
%
	Note that from Lemma~\ref{lm:F} there exists $N_2\in\NB$ such that 
	$$ \big| \mu _N(A)-\mu (A)\big| <\tfrac{\eps}{2} $$
	for any $N\geq N_2$ and measurable set $A$. By taking $N_1$ larger so that the previous bound holds, using Lemma~\ref{lm:uniform}(a) and the fact that $\FL^{s,p}(\T)$, $\FL^{\s,p}(\T)$ have full $\mu$-measure, we have
	\begin{align*}
	\mu \big(\FL^{s,p}(\T)  \setminus \Sigma_{T,\eps}^s(A)\big) & \leq \mu_{N_1} \big( \FL^{\s,p}(\T) \setminus \Omega^\s_{N_1} (T,\tfrac{\eps}{2}, A) \big) + \tfrac{\eps}{2}  <  \eps,
	\end{align*}
	which completes the proof.
\end{proof}

We can now show Theorem~\ref{th:invariance}.
\begin{proof}[Proof of Theorem~\ref{th:invariance}]
	This proof follows the approaches in \cite{Tzv08,NahOhBelletSta12}.
	We first establish almost sure global well-posedness of $\mathcal{G}$-gKdV. Define an increasing sequence $\{s_j\}_{j\in\NB}$ by $s_1=\frac12 (s_*+1 - \frac1p)$ and $s_{j+1}=\frac12 (s_j+1-\frac1p)$, which converges to $1 - \frac1p$ as $j\to\infty$. 
	Fix $0<\eps \leq 1$ and let $T_j = 2^j$, $\eps_j = 2^{-j} \eps$, $j\in\NB$. For $\Sigma_{T_j, \eps_j}^{s_j}(2^k)$ as defined in Proposition~\ref{prop:almostGWP}, with $s=s_j$ and $\s= s_{j+1}$, let 
	$$\Sigma_\eps = \bigcap_{j=1}^\infty \Sigma_{T_j,\eps_j}^{s_j} = \bigcap_{j=1}^\infty \Big(\bigcup_{k=1}^\infty \Sigma_{T_j, \eps_j}^{s_j}(2^k) \Big).$$
	Lastly, let $\Sigma = \bigcup_{n=1}^\infty \Sigma_{\frac1n}$.
	
	First note that $\Sigma \subset \bigcap_{s<1 - \frac1p} \FL^{s,p}(\T)$. Let $u_0 \in \Sigma$, then for any $j\in\NB$, we have $u_0 \in \Sigma_{T_j, \eps_j}^{s_j}(2^k)$ for some $\eps = \frac1n$ and $k\in\NB$. Hence, by Proposition~\ref{prop:almostGWP}, there exists a solution $u \in C\big([-T_j, T_j]; \FL^{s_j,p}(\T)\big) \cap Z^{s_j,\frac12}_p(T_j)$ of $\mathcal{G}$-gKdV with $u(0) = u_0$. By uniqueness of local solutions in $Z^{s,\frac12}_p(T)$, we obtain a unique global solution $u\in \bigcap_{s<1- \frac1p} C\big(\R;\FL^{s,p}(\T)\big)$
. Moreover, since $\Sigma_{T_j, \eps_j}^{s_j} (2^k)$ is closed in $\FL^{s_1,p}(\T)$ and $\mu (\Sigma_\eps ^c)\leq \sum _{j\in \NB}\mu \big( (\Sigma^{s_j}_{T_j,\eps_j})^c\big) <\eps$, $\Sigma$ is $\mu$-measurable and $\mu(\Sigma^c)=0$. 
	
	We now establish that $\Phi(t) \Sigma = \Sigma$ for any $t\in\R$, where $\Phi(t):u_0\mapsto u(t)$ denotes the solution map of $\mathcal{G}$-gKdV defined above. Fix $\tau\in\R$. It suffices to show that $\Phi(\tau) \Sigma \subset \Sigma$, as the other inclusion follows from this and the reversibility of the flow. It suffices to show that $\Phi(\tau) \Sigma_\eps \subset \Sigma_\eps$, $\eps =\frac1n$ for each $n\in\NB$. We actually establish that if $|\tau|\leq T_k$ for some $k\in\NB$, then for every $i\in\NB$, $\Phi(\tau) \Sigma_{T_{j}, \eps_{j}}^{s_{j}} \subset \Sigma_{T_i, \eps_i}^{s_i}$ for $j=\max(i+2, k+1)$, from which the intended result follows. Let $u_0 \in \Sigma_{T_{j}, \eps_{j}}^{s_{j}}$, then there exists $A\in 2^{\NB}$ such that $u_0 \in \Sigma_{T_{j}, \eps_{j}}^{s_{j}}(A)$. From Proposition~\ref{prop:almostGWP}, there exists a solution $u(t)$ of $\mathcal{G}$-gKdV for $|t| \leq T_{j}$ satisfying
	\begin{equation*}
		\|u(t)\|_{\FL^{s_{j}, p}} \leq 2 \sqrt{2} A C_0 C_{s_{j+1}} \big(\log \tfrac{T_{j}}{\eps_{j}} \big)^\frac12, \qquad |t| \leq T_{j}.
	\end{equation*}
	Note that $u_\tau(t) = u(\tau +t)$ is a solution of $\mathcal{G}$-gKdV with $u_\tau(0) = u(\tau )=\Phi(\tau)u_0$, which belongs to $C\big([-T_{j-1},T_{j-1}] ; \FL^{s_j,p}(\T) \big) \cap Z^{s_j,\frac12}_{p}(T_{j-1})$, because $k\leq j-1$ and then $|t + \tau| \leq T_{j-1} + T_k \leq T_j$. Since the above estimate holds for $u_\tau(t)$ if $|t| \leq T_{j-1}$, from Lemma~\ref{lm:approx}(b), it follows that
	\begin{equation*}
		\| \Phi_N(t) \Phi(\tau) u_0 \|_{\FL^{s_{j-1},p}} \leq 6 \sqrt{2} A C_0 C_{s_{j+1}} \big(\log\tfrac{T_j}{\eps_j}\big)^\frac12, \qquad |t| \leq T_{j-1},
	\end{equation*}
	for any $N\geq N_0$.
	Since $i\leq j-2$ and $\frac{T_j}{\eps_j}\leq \big( \frac{2T_i}{\eps_i}\big) ^{j/i}$ for $0<\eps\leq 1$, we get that
	\begin{align*}
		\|\Phi_N (t) \Phi(\tau) u_0 \|_{\FL^{s_{i+1},p}} \leq 6\sqrt{2j/i} A C_0 C_{s_{j+1}} \big( \log\tfrac{2 T_i}{\eps_i}\big)^\frac12, \qquad |t| \leq T_{i+1}.
	\end{align*}
	Consequently, by choosing $\tilde{A} \in 2^\NB$ such that $6\sqrt{2j/i} A C_0 C_{s_{j+1}} \leq \tilde{A} C_{s_{i+1}}$ and $N_1(\tilde{A}) \geq N_0$, and applying Lemma~\ref{lm:uniform}(c), we conclude that $\Phi(\tau) u_0 \in \Sigma_{T_i, \eps_i}^{s_i}(\tilde{A})$. The group property of $\Phi (t)$ follows from uniqueness of local solutions in $Z^{s,\frac12}_p(T)$.
	
	Before showing the invariance of $\mu$ under the flow map $\Phi(t)$, we show that $\Phi(t)$ is $\mu$-measurable for every $t\in\R$. It suffices to show the continuity of the map in the topology induced by $\FL^{s_1,p}(\T)$. Fix $t\in\R$ and $u_0\in\Sigma$. Consider a sequence $\{u_{0,k}\}_{k\in\NB} \subset \Sigma$ converging to $u_0$ in $\FL^{s_1,p}(\T)$. Let $j\in\NB$ such that $|t| \leq T_j$. Then, $u_0 \in \Sigma_{T_j, \eps_j}^{s_j}(A)$ for some $\eps = \frac1n$ and some $A$. By Proposition~\ref{prop:almostGWP}, we have
	$$\sup_{|\tau| \leq T_j} \|\Phi(\tau) u_0 \|_{\FL^{s_j,p}} \leq 2 \sqrt{2} A C_0 C_{s_{j+1}}\big( \log \tfrac{T_j}{\eps_j} \big)^\frac12 = : \Lambda.$$
	Let $\tau_0$ be the local time of existence for data of size $2\Lambda$ in $\FL^{s_1,p}(\T)$.
	From the Lipschitz continuity of the solution map, we obtain
	\begin{align*}
		\|\Phi(t) u_0 - \Phi(t)u_{0,k} \|_{\FL^{s_1,p}} \leq C^{[\frac{|t|}{\tau_0}]} \|u_0 - u_{0,k} \|_{\FL^{s_1,p}},
	\end{align*} 
	as long as the right-hand side is bounded by $\Lambda$, which holds for $k$ large enough. Consequently, by taking $k\to\infty$, we conclude that $\Phi(t) u_{0,k} \to \Phi(t) u_0$ in $\FL^{s_1,p}(\T)$.
	
	It remains to show the invariance of the Gibbs measure $\mu$ under the flow $\Phi(t)$ of $\mathcal{G}$-gKdV \eqref{gauged}. Having established the flow property of $\Phi(t)$, it suffices to show that for all $G \in L^1 \big( \FL^{s_1,p}(\T), d\mu \big)$ and $t\in \R$, we have
	\begin{equation}\label{invariance}
	\int_\Sigma G\big( \Phi(t) u \big) d\mu(u) = \int_\Sigma G(u) d \mu(u).
	\end{equation}
	Moreover, it suffices to show \eqref{invariance} for $G$ in a dense subset of $L^1\big( \FL^{s_1,p}(\T), d\mu \big)$. In particular, we choose this set $\mathcal{H}$ as 
	the set of continuous and bounded functions on $\FL^{s_1,p}(\T)$. 
	Fix $G\in\mathcal{H}$, $t\in \R$ and $\kappa>0$. We have the following
	\begin{align*}
	\bigg| \int_\Sigma G\big(\Phi(t) u \big) d\mu(u) - \int_\Sigma G(u) d\mu(u) \bigg| & \leq \bigg|\int_\Sigma G\big(\Phi(t)u \big) d\mu(u) - \int_\Sigma G\big(\Phi(t)u\big) d\mu_N(u)\bigg| \\
	&\quad + \bigg| \int_\Sigma G\big(\Phi(t) u \big) d\mu_N(u) - \int_\Sigma G\big(\Phi_N(t)u\big) d\mu_N(u)  \bigg| \\
	&\quad + \bigg| \int_\Sigma G\big(\Phi_N(t)u \big) d\mu_N(u) - \int_\Sigma G(u) d\mu_N(u) \bigg| \\
	&\quad + \bigg| \int_\Sigma G(u) d\mu_N(u) - \int_\Sigma G(u) d\mu(u) \bigg| \\
	& = \I + \II + \III +\IV .
	\end{align*}
	From Lemma~\ref{lm:F}, we have 
	\begin{align*}
	\int \tilde{G}(u) d\mu_N(u) - \int \tilde{G}(u) d\mu(u) = \int \tilde{G}(u)\Big( \frac{F_N(u)}{\| F_N\|_{L^1(d\rho )}}-\frac{F(u)}{\| F\|_{L^1(d\rho )}}\Big) d\rho (u) \to 0 , \quad N\to \infty
	\end{align*}
	for every bounded measurable function $\tilde{G}$ on $\FL^{s_1,p}(\T)$. Consequently, since $G$ is bounded and continuous and $\Phi(t)$ is measurable, there exists $N_0\in\NB$ such that $\I + \IV < \frac{\kappa}{2}$, for $N\geq N_0$.
	From Proposition~\ref{inv_trunc}, the measure $\mu_N$ is invariant under the flow $\Phi_N(t)$, thus $\III=0$.
	It only remains to estimate $\II$. For $0<\eps \leq \frac12$, consider the set $\Sigma (t,\eps )=\Sigma^{s_2}_{1+|t|,\eps}(1)\subset \FL^{s_3,p}(\T)$. From Lemma~\ref{lm:F}, there exists $N_1\in\NB$ such that $\mu_N(\Sigma (t,\eps)^c) < \mu(\Sigma (t,\eps )^c) + \eps$ for $N\geq N_1$. Since $\mu(\Sigma (t,\eps )^c) < \eps$ by Proposition~\ref{prop:almostGWP}, we see that
	\begin{align*}
	\bigg| \int_{\Sigma \setminus \Sigma (t,\eps)} G\big(\Phi(t) u \big) d\mu_N(u) - \int_{\Sigma\setminus \Sigma (t,\eps)} G\big(\Phi_N(t)u \big) d\mu_N(u) \bigg| 
	\leq 2 \|G\|_{L^\infty} \big( \mu\big(\Sigma (t,\eps)^c\big) + \eps \big) 
	& < \frac\kappa4,
	\end{align*}
	for $N\geq N_1$ and by choosing $\eps \leq \frac{\kappa}{16\|G\|_{L^\infty}}$. In order to estimate the contribution restricted to $\Sigma(t,\eps)$, we want to exploit the continuity of $G$. 
	For $u_0\in \Sigma \cap \Sigma(t,\eps)$, from Proposition~\ref{prop:almostGWP} and uniqueness, we have
	\begin{equation*}
	\|u_0\|_{\FL^{s_2,p}} , \ \|\Phi (s) u_0 \|_{\FL^{s_2,p}} \leq 2\sqrt{2}C_0C_{s_3}\big( \log \tfrac{1+|t|}{\eps}\big) ^{\frac12},\qquad |s|\leq 1+|t|.
	\end{equation*}
In particular, the set $\{ \Phi(t) u_0 : u_0\in \Sigma \cap \Sigma (t,\eps )\}$ is bounded in $\FL^{s_2,p}(\T)$ and thus precompact in $\FL^{s_1,p}(\T )$, which implies that $G$ is uniformly continuous on this set. Next, from Lemma~\ref{lm:approx}(b) we have
	\begin{align*}
	\big\| \Phi (t) u_0 - \PP_{\leq N} \Phi_N(t) u_0 \big\|_{\FL^{s_1,p}} &\leq C(t,\eps ) N^{-(s_2-s_1)}
	\end{align*}
	for any $N$ large enough. Thus, it follows that 
	\begin{align*}
	\big\| \Phi(t) u_0 - \Phi_N(t) u_0 \big\|_{\FL^{s_1,p}} & \leq \big\| \Phi(t) u_0 - \PP_{\leq N}\Phi_N(t) u_0 \big\|_{\FL^{s_1,p}} + \big\| \PP_{>N} \Phi_N(t) u_0 \big\|_{\FL^{s_1,p}} \\
	& \leq C(t,\eps ) N^{-(s_2-s_1)} + N^{-(s_2-s_1)} \| u_0 \|_{\FL^{s_2,p}} \\
	& \leq C(t,\eps ) N^{-(s_2-s_1)} .
	\end{align*}
Hence, there exists $N_2\in \NB$ depending on $t,\eps$ such that
	\begin{equation*}
	\big| G(\Phi(t)u_0) - G(\Phi_N(t)u_0) \big| < \frac{\kappa}{4}
	\end{equation*}
	for $N\geq N_2$ and $u_0\in \Sigma \cap \Sigma (t,\eps )$, and we can estimate the remaining piece of $\II$,
	\begin{align*}
	\bigg| \int_{\Sigma \cap \Sigma(t,\eps)} G\big(\Phi(t)u\big) d\mu_N(u) - \int_{\Sigma \cap \Sigma(t,\eps)} G\big(\Phi_N(t) u \big) d\mu_N(u) \bigg| & \leq \int \frac{\kappa}{4} \,d\mu_N(u)
	 = \frac\kappa4.
	\end{align*}
	Consequently, we have that $\II <\frac{\kappa}{2}$ for $N\geq \max (N_1,N_2)$. Combining all the estimates, we obtain
	\begin{align*}
	\bigg| \int_\Sigma G\big( \Phi(t)u \big) d\mu(u) - \int_\Sigma G\big( u \big) d\mu(u) \bigg| < \kappa.
	\end{align*}
	Since $\kappa$ is arbitrarily small, we obtain \eqref{invariance}, as intended.
\end{proof}

Lastly, we establish the invariance of the Gibbs measure $\mu$ under the flow $\Psi(t)$ of the original gKdV equation \eqref{gkdv}.
\begin{proof}[Proof of Theorem~\ref{th:invariance_gkdv}]
	Let $\Sigma$ be the subset of $\bigcap_{s<1 - \frac1p} \FL^{s,p}(\T)$ constructed in Theorem~\ref{th:invariance} and denote by $T(y)$, for $y\in \T$, the spatial translation operator $f(x)\mapsto f(x - y)$. Note that $\Sigma$ is invariant under $T(y)$. Consequently, we can establish the global-in-time dynamics on $\Sigma$ for the gKdV equation \eqref{gkdv} with the solution map $\Psi(t)$ satisfying the flow property as \eqref{flowproperty}; see Appendix~\ref{ap:gauge} for the definition of $\Psi (t)$ and the proof of the group property of it.

	It remains to prove the invariance of the Gibbs measure \eqref{intro_invariance}.\footnote{We would like to thank Terence Tao and Rowan Killip for suggesting this argument.}
	Let $\haar$ denote the Haar measure on $\T$. Fix $A\subset \Sigma$ and $t\in\R$. Using the invariance of $\mu$ under $T(y)$,\footnote{To see this, we first observe that
\[ T(y)\Big[ \sum _{n\in \mathbb{Z}_*}\frac{g_n(\omega )}{|n|}e^{inx}\Big] = \sum _{n\in \mathbb{Z}_*}\frac{e^{-iny}g_n(\omega )}{|n|}e^{inx}.\]
Then, from the invariance of complex Gaussians under rotations, we see that the Gaussian measure $\rho$ is invariant under $T(y)$.
This implies the invariance of the Gibbs measure $\mu$ under $T(y)$, since the density $F(u)=\1_{\{\|u\|_{L^2} \leq R\}} e^{- \frac{1}{k+1} \int_\T u^{k+1} dx }$ is invariant under $T(y)$.} the fact that $T(y)$ and $\Psi(t)$ commute and Fubini's Theorem, we have
\begin{align*}
	\mu\big( \Psi(-t)A \big) & = \int_\T \mu \big( T(-y) \Psi(-t) A\big) \, d\haar(y) \\
	& = \int_\T \int_\Sigma \1_A\big( T(y) \Psi(t)u_0\big) \, d\mu(u_0) \, d\haar (y)\\
	& = \int_\Sigma \int_\T \1_A \bigg[ T\Big( y \pm k\int_0^t \PP_0(\Phi(t') u_0)^{k-1} dt' \Big) \Phi(t) u_0 \bigg] \, d\haar(y) \, d\mu(u_0).
\end{align*}
From the translation invariance of $\haar$, Fubini's Theorem and the fact that $\Phi(t)$ commutes with $T(y)$, we have that
\begin{align*}
	\mu\big( \Psi(-t)A \big) & = \int_\Sigma \int_\T \1_A \big( T(y) \Phi(t) u_0 \big) \, d\haar(y) \, d\mu(u_0)\\
	& = \int_\T \mu\big( T(-y) \Phi(-t) A\big) \, d\haar(y) .
\end{align*}
Since $\mu$ is invariant under $T(y)$ and under the flow map $\Phi(t)$ of \eqref{gauged} from Theorem~\ref{th:invariance}, we get $\mu\big( \Psi(-t)A \big) = \mu\big(\Phi(-t) A \big) = \mu(A)$, as intended.
\end{proof}

\begin{appendix}

	\section{Gauge transformation and solution map for gKdV}\label{ap:gauge}
	We start by establishing continuity of the (inverse) gauge transformation.
	\begin{lemma}\label{lem:cont-G}
		The (inverse) gauge transformation in \eqref{gauge} is a continuous map on \\
		$C\big([-T,T]; \FL^{s,p}(\T)\big)$ given that $1\leq p < \infty$ and $s> 1 - \frac1p - \frac{1}{k-1}$.
	\end{lemma}
	\begin{proof}
		Let $u$ be any function in $C\big([-T,T]; \FL^{s,p}(\T)\big)$. Consider a sequence $\{u_m\}_{m\in\NB}$ in $C\big([-T,T]; \FL^{s,p}(\T)\big)$ converging to $u$ and fix $t\in[-T,T]$. Then, 
		\begin{align*}
		&\big\|\mathcal{G}_{0,t}\big(u(t)\big) - \mathcal{G}_{0,t}\big(u_m(t)\big) \big\|_{\FL^{s,p}} \\
		& \phantom{XXX}= \big\| \jb{n}^s\big(e^{in k \int_0^t \PP_0(u^{k-1} (t')) dt'} \ft{u}(t,n) - e^{in k\int_0^t \PP_0(u_m^{k-1} (t')) dt'} \ft{u}_m(t,n) \big) \big\|_{\l^p_n} \\
		&\phantom{XXX} \leq 2 \big\| \1_{|n| > N} \jb{n}^s \ft{u}(t,n) \big\|_{\l^p_n} + \| u(t) - u_m(t) \|_{\FL^{s,p}} \\
		&\phantom{XXXXXX} + \| u(t) \|_{\FL^{s,p}} \big\| \1_{|n| \leq N} \big( e^{in k\int_0^t \PP_0(u^{k-1} (t')) dt'}- e^{ink \int_0^t \PP_0(u_m^{k-1} (t')) dt'} \big)\big\|_{\l^\infty_n}.
		\end{align*}
		The first two terms on the right-hand side of the estimate converge to zero as $N\to\infty$ and $m\to\infty$, thus it only remains to consider the last one. Using the mean value theorem, we have
		\begin{align*}
		\big\| \1_{|n| \leq N} \big( e^{ink\int_0^t \PP_0(u^{k-1} (t')) dt'}- e^{ink\int_0^t \PP_0(u_m^{k-1} (t')) dt'} \big)\big\|_{\l^\infty_n} & \leq |t| N \| u^{k-1} - u_m^{k-1} \|_{C_{|t|}L^1}.
		\end{align*}
		Since $\FL^{s,p}(\T) \embeds L^{k-1}(\T)$ for $s>1 - \frac1p - \frac{1}{k-1}$, then the above quantity converges to zero for each fixed $N$, establishing the continuity of $\mathcal{G}_{0,t}$. An analogous proof works for $\mathcal{G}_{0,t}^{-1}$.
	\end{proof}
	
	Following the argument in \cite{GH}, we establish the following result for the (inverse) gauge transformation in \eqref{gauge}.
	\begin{proposition}
		Let $1 \leq p < \infty$ and $s> 1 - \frac1p - \frac{1}{k-1}$. Then, the (inverse) gauge transformation in \eqref{gauge} is not uniformly continuous on arbitrarily small balls of $C\big([-T,T]; \FL^{s,p}(\T)\big)$ centered at the origin.
	\end{proposition}
	\begin{proof}
		Let $R>0$ and $N\in\NB$. Define $\{u_{N,j}\}_{N\in\NB}$ for $j=1,2$ as follows
		\begin{align*}
		u_{N,1} (t,x) &= RN^{-s} (e^{iNx} + e^{-iNx}) + N^{-\frac{1}{k-1}} (e^{iMx} + e^{-iMx}), \\
		u_{N,2} (t,x) &= RN^{-s} (e^{iNx} + e^{-iNx}),
		\end{align*}
		with $M=0$ for $k$ even, and $M=1$ for $k$ odd. Note that 
		\begin{align*}
		\|u_{N,1}\|_{C_T\FL^{s,p}} \les R, 
		\end{align*}
		for $N$ large enough, and $\|u_{N,2}\|_{C_T\FL^{s,p}} \sim R$. Moreover,
		\begin{align*}
		\|u_{N,1} - u_{N,2}\|_{C_T \FL^{s,p}} \sim N^{-\frac{1}{k-1}} \to 0, 
		\end{align*}
		as $N\to\infty$.
		Using the mean value theorem, we obtain
		\begin{align*}
		\|\mathcal{G}_{0,t} (u_{N,1}) - \mathcal{G}_{0,t}(u_{N,2}) \|_{C_T \FL^{s,p}} & \geq TN \Big| \int_\T \big( u^{k-1}_{N,1}(x) - u^{k-1}_{N,2}(x) \big) \, dx \Big|.
		\end{align*}
		Calculating $\int_\T (u^{k-1}_{N,1} - u^{k-1}_{N,2})\, dx$, we have
		\begin{align*}
		 \sim \sum_{\substack{1\leq j \leq k-1\\0\leq l \leq k-1-j \\ 0 \leq m \leq j}} {k-1 \choose j} 
		 N^{-s (k-1-j) - \frac{j}{k-1}} \int_\T e^{iNx(k-j-2l-1) + iMx(j-2m)}, 
		\end{align*}
		thus the nonzero contributions correspond to the choices of indices satisfying $k-1 - j=2l$ and $M(j-2m)=0$ , since $N\gg M$.
		Consequently, we see that the quantity is dominated by the contribution at $j=k-1$, therefore
		\begin{align*}
		\|\mathcal{G}_{0,t} (u_{N,1}) - \mathcal{G}_{0,t}(u_{N,2}) \|_{C_T \FL^{s,p}} & \ges 1, 
		\end{align*}
		which does not decay as $N\to\infty$.
	\end{proof}
	
	We now focus on the solution map of gKdV \eqref{gkdv}. We can define the map $\Psi(s,t)$ for $t,s\in \R$ as
	\begin{align*}
		\Psi(s,t) u_0 = \big[ \Phi(t-s) u_0\big] \bigg( x\pm k\int_s^t \PP_0 \big( \Phi(t'-s) u_0 \big)^{k-1} \, dt' \bigg) ,
	\end{align*}
	which is a solution of gKdV \eqref{gkdv} at time $t$, with initial data $u_0$ at time $s$. Since $\Psi(s,t) = \Psi(0,t-s)$, we can denote the solution map of gKdV \eqref{gkdv} at time $t$ as $\Psi(t) := \Psi(0,t)$. The following lemma establishes that the solution map $\Psi(t)$ satisfies the group property.
	
	\begin{lemma}
		For any $t,s\in\R$ we have that $\Psi(t+s) = \Psi(t)\Psi(s)$.
	\end{lemma}
	\begin{proof}
		Let $u_0 \in \FL^{s,p}(\T)$ and $t,s\in\R$. From the definition of $\Psi$, we have
		\begin{align*}
			\Psi(s+t) u_0 = \big[ \Phi(s+t) u_0\big] \bigg( x \pm k \int_0^{s+t} \PP_0 \big(\Phi(t')u_0\big)^{k-1}  \, dt' \bigg).
		\end{align*}
		Using the group property of $\Phi$ and a change of variables, we obtain
		\begin{align*}
			\Psi(s)\Psi(t) u_0& = \Psi(s) \bigg[ \big[\Phi(t) u_0\big] \bigg( x \pm k\int_0^t \PP_0 \big(\Phi(t') u_0 \big)^{k-1} \, dt' \bigg) \bigg] \\
			& = \big[ \Phi(s+t) u_0\big]  \bigg(  x \pm k \int_0^t \PP_0\big(\Phi(t')u_0 \big)^{k-1} \, dt' \pm k \int_0^s \PP_0 \big(\Phi(t+t') u_0 \big)^{k-1} \, dt' \bigg)\\
			& = \big[ \Phi(s+t) u_0\big] \bigg( x \pm k \int_0^{t+s} \PP_0 \big(\Phi(t') u_0\big)^{k-1} \, dt' \bigg),
		\end{align*}
		which is equal to $\Psi(s+t)u_0$, establishing the group property of the map.
	\end{proof}

	\section{Lifting the mean zero condition}\label{ap:mean}
	In this section, we clarify how to construct the solution map for \eqref{gauged} without restricting to mean zero initial data. 
	We first consider the set $\FL^{s,p}_\al(\T)$ of functions in $\FL^{s,p}(\T)$ with prescribed mean $\al \in \R$, and define the translation $\tau_\al[u] := u - \al$. For $u_0 \in \FL^{s,p}_\al(\T)$, $v_0 = \tau_\al[u_0] \in \FL^{s,p}_0(\T)$ and we consider the following Cauchy problem
	\begin{equation}\label{gauged-alpha}
		\left\{ \begin{aligned} &\dt v + \dx^3 v = k \PP \big( (v+\alpha )^{k-1}\big) \dx v,\\
			&v|_{t=0}=v_0. \end{aligned} \right.
	\end{equation}
	Since conservation of mean still holds for solutions of \eqref{gauged-alpha}, $v$ has mean zero and we can apply the nonlinear estimates in Proposition~\ref{prop:nonlinear} to $v$.%
	\footnote{Precisely, we need the estimates of $\PP (v^l)\dx v$, $l=k-1,k-2,\dots ,1$, for mean zero $v$. Proposition~\ref{prop:nonlinear} treats the case $l\geq 3$, while the $l=2$ case can be found in \cite{Ch21}, Proposition~5, which holds for $2<p<\infty$ and $s>\max (\frac12, \frac34 -\frac1p)$. For $l=1$, by adapting the proof of Proposition~\ref{prop:nonlinear} we can easily see that the required estimate is available at least for $2<p<\infty$ and $s>\frac12$. In fact, there are only two frequencies $n_0$, $n_1$, and we can treat two possibilities $|n_0|\gg |n_1|$ and $|n_0|\sim |n_1|$ by following the argument for {\bf Case~1.1} and {\bf Case~2.1}, respectively. Note that the most restrictive condition on $s$ is that for $l=k-1$, which is the same as the one imposed in Theorem~\ref{th:lwp} for the mean-zero case.}
	Following the proof of Theorem~\ref{th:lwp}, we prove that \eqref{gauged-alpha} is locally well-posed in $\FL^{s,p}_0(\T)$, with local time of existence $T\sim (1+\| v_0\|_{\FL^{s,p}}+|\alpha |)^{-\frac{k-1}{\theta}}$. Let $\Phi^\al_0(t)$ be the obtained solution map. We can now define the local-in-time flow $\Phi_\al(t)$ of the gauged gKdV equation \eqref{gauged} on $\FL^{s,p}_\al(\T)$ as
	$$\Phi _\alpha (t)=\tau _{-\alpha}\circ \Phi _0^\alpha (t)\circ \tau_\alpha, $$
	and the flow $\cj{\Phi}(t)$ of \eqref{gauged} on $\FL^{s,p}(\T)$ as 
	$$\cj{\Phi}(t) = \Phi_\alpha (t)\quad \text{on $\FL^{s,p}_\alpha$, for each $\alpha \in \mathbb{R}$.} $$
	Similarly to the mean zero case, this solution map $\cj{\Phi}(t)$ is still locally Lipschitz continuous.
	\begin{proposition}\label{prop:lwp}
		
		For any $R>0$ there exists $T\sim (1+R)^{-\frac{k-1}{\theta}}$ such that the flow $\overline{\Phi}(t)$ can be defined on $B_R:=\{ u_0\in \FL^{s,p}:\| u_0\|_{\FL^{s,p}}\leq R\}$ for $|t|\leq T$.
		Moreover, for any $u_0,\wt{u}_0\in B_R$ we have
		\[ \| \overline{\Phi}(\cdot )u_0\|_{Z^{s,\frac12}_p(T)}\lesssim \| u_0\|_{\FL^{s,p}},\qquad \| \overline{\Phi}(\cdot )u_0-\overline{\Phi}(\cdot )\wt{u}_0\|_{Z^{s,\frac12}_p(T)}\lesssim \| u_0-\wt{u}_0\|_{\FL^{s,p}},\]
		where the implicit constants are independent of $R$ and the means of $u_0,\wt{u}_0$.
	\end{proposition}
	
	\begin{proof}
		The flow $\cj{\Phi}(t)$ is well-defined on $B(R)$ since each flow $\Phi_\al(t)$ is defined on $B_R \cap \FL^{s,p}_\al(\T)$ for $|t| \leq T \sim ( 1 + R +|\alpha |)^{-\frac{k-1}{\theta}}$ and $|\al| = |\PP_0 u_0| \leq \|u_0\|_{\FL^{s,p}} \leq R$ on $B_R \cap \FL^{s,p}_\al(\T)$.
		From the local theory in $\FL^{s,p}_0(\T)$, we have
		\[ \| \Phi _0^\alpha (\cdot )\tau_\alpha u_0\| _{Z^{s,\frac12}_p(T)}\lesssim \| \tau_\alpha u_0\|_{\FL^{s,p}}\leq \| u_0\|_{\FL^{s,p}}\]
		for $u_0\in \FL^{s,p}_\alpha $, and hence 
		\[ \| \overline{\Phi}(\cdot )u_0\| _{Z^{s,\frac12}_p(T)}\leq \| \Phi _0^\alpha (\cdot )\tau_{\alpha}u_0\| _{Z^{s,\frac12}_p(T)}+C|\alpha |\lesssim \| u_0\|_{\FL^{s,p}},\]
		where the implicit constants are uniform in $\alpha$.
		
		To prove the Lipschitz bound, let $u_0,\wt{u}_0\in B_R$ be two initial data with means $\alpha$ and $\wt{\alpha}$, respectively.
		Note that $|\alpha |\leq R$, $|\wt{\alpha}|\leq R$ and $|\alpha -\wt{\alpha}|\leq \| u_0-\wt{u}_0\|_{\FL^{s,p}}$.
		Let $v(t)=\Phi ^\alpha_0(t)[\tau _\alpha u_0]$ and $\wt{v}(t)=\Phi ^{\wt{\alpha}}_0(t)[\tau _{\wt{\alpha}}\wt{u}_0]$ be the corresponding solutions of \eqref{gauged-alpha} for $\al$ and $\wt{\al}$, respectively.
		From the local well-posedness of \eqref{gauged-alpha} in $\FL^{s,p}_0$, we have $\| v\|_{Z^{s,\frac12}_p(T)}$, $\| \wt{v}\|_{Z^{s,\frac12}_p(T)}\lesssim R$ for $T\sim (1+R)^{-\frac{k-1}{\theta}}$.
		Let $w=v-\wt{v}$. Then, we can use the nonlinear estimates in Proposition~\ref{prop:nonlinear} to show that
		\[ \| w\| _{Z^{s,\frac12}_p(T)}\leq C\| \tau_\alpha u_0 -\tau_{\wt{\alpha}}\wt{u}_0\|_{\FL^{s,p}}+CT^\theta R^{k-1}\big( |\alpha -\wt{\alpha}|+\| w\| _{Z^{s,\frac12}_p(T)}\big) .\]
		Replacing $T\sim (1+R)^{-\frac{k-1}{\theta}}$ if necessary, we have
		\begin{align*}
			\| w\| _{Z^{s,\frac12}_p(T)}&\lesssim \| \tau_\alpha u_0 -\tau_{\wt{\alpha}}\wt{u}_0\|_{\FL^{s,p}}+|\alpha -\wt{\alpha}| \\
			&\leq \| u_0 -\wt{u}_0\|_{\FL^{s,p}}+2|\alpha -\wt{\alpha}|.
		\end{align*}
		Therefore, 
		\begin{align*}
			\| \overline{\Phi}(\cdot )u_0-\overline{\Phi}(\cdot )\wt{u}_0\|_{Z^{s,\frac12}_p(T)}&=\| \tau_{-\alpha}v-\tau_{-\wt{\alpha}}\wt{v}\|_{Z^{s,\frac12}_p(T)}\\
			&\lesssim \| w\|_{Z^{s,\frac12}_p(T)}+|\alpha -\wt{\alpha}|\\
			&\lesssim \| u_0 -\wt{u}_0\|_{\FL^{s,p}}+|\alpha -\wt{\alpha}|\\
			&\lesssim \| u_0 -\wt{u}_0\|_{\FL^{s,p}}.\qedhere
		\end{align*}
	\end{proof}
			
	\begin{remark}\rm
		The same conclusions as in Proposition~\ref{prop:lwp} hold for the flow $\overline{\Phi}_N(t)$ of the truncated equation
		\begin{equation}\label{gauged-N}
			\left\{ \begin{aligned} &\dt u_N + \dx^3 u_N = k\PP_{\leq N} \Big[ \PP \big( (\PP_{\leq N}u_N)^{k-1}\big) \dx \PP_{\leq N}u_N\Big] ,\\
				&u_N|_{t=0}=u_0, \end{aligned} \right.
		\end{equation}
		with $u_0 \in \FL^{s,p}(\T)$,
		and the result is uniform in $N$.
		In fact, since the mean is still conserved for solutions to \eqref{gauged-N}, we can define $\Phi_{N,\alpha}(t)=\overline{\Phi}_N(t)|_{\FL^{s,p}_\alpha}$ by $\Phi_{N,\alpha}(t)=\tau _{-\alpha}\circ \Phi ^\alpha_{N,0}(t)\circ \tau _\alpha$, where $\Phi^\alpha_{N,0}$ is the flow of the equation for $v_N=\tau_\alpha u_N$ given by
		\begin{equation}\label{gauged-alpha-N}
			\left\{ \begin{aligned} &\dt v_N + \dx^3 v_N = k\PP_{\leq N} \Big[ \PP \big( (\PP_{\leq N}v_N+\alpha )^{k-1}\big) \dx \PP_{\leq N}v_N\Big] ,\\
				&v_N|_{t=0}=v_0, \end{aligned} \right. 
		\end{equation}
	with $v_0 = \tau_\al u_0 \in \FL^{s,p}_0(\T)$.
		The argument for \eqref{gauged-alpha-N} is analogous to that for \eqref{gauged-alpha} and uniform in $N$. 
	\end{remark}
		
		Now, the a.s. global well-posedness of \eqref{gauged} and of \eqref{gkdv} without prescribing mean zero and the invariance of the Gibbs measure follow from the approach in Section~\ref{sec:inv}.

	\section{Invariance in the threshold case}\label{ap:threshold}
	In this section, we outline the proof of Theorem~\ref{th:invariance} (b) when $k=5$ and $R = \|Q\|_{L^2(\R)}$.
	For $\dl \geq 0$, we introduce the notation 
	\begin{align*}
		F_\dl(u) & = e^{\frac{1}{6} \int_\T u^{6} \, dx} \1_{\{\| u \|_{L^2} \leq R - \dl\}},
	\end{align*}
	and let $\mu_\dl$ denote the Gibbs measure given by
	\begin{align*}
		d \mu_\delta (u) = Z_\delta^{-1} F_\dl(u) \ d\rho(u).
	\end{align*}
	From Theorem~\ref{th:focusing}, we have that
	\begin{align*}
		F_\dl(u) & \in L^q(d\rho), \quad 1\leq q < \infty, ~~\dl>0, \\
		F_0(u) & \in L^1(d\rho).
	\end{align*}
	In addition, fix $2<p<\infty$ and set $s=s_1= \frac12(s_*(p)+1-\frac1p)$, $\mathcal{B}$ the Borel $\sigma$-field on~$\FL^{s,p}(\T)$.
	
	We will use the following approximation property between $\mu_\dl$ and $\mu_0$. Note that $0\leq F_\dl(u) \leq F_0(u)$ for $\dl>0$ and $F_\dl (u)\to F_0(u)$ for any $u \in \FL^{s,p}(\T)$. Therefore, by the dominated convergence theorem, we have that $F_\dl \to F_0$ in $L^1(d\rho)$, from which we obtain
	\begin{equation}\label{conv0}
		\sup _{A\in \mathcal{B}}\big| \mu_\delta (A) - \mu_0(A)\big| \to 0\qquad \text{ as } \dl \to 0.
	\end{equation}
	For $\dl>0$, the argument in Section~\ref{sec:inv} holds. Consequently, by Theorem~\ref{th:invariance}, we can construct a set $\Sigma_\dl$ of full $\mu_\dl$-measure and a global-in-time flow $\Phi_\dl(t)$ for \eqref{gauged} on $\Sigma_\dl$. 
	
	Now, we establish Theorem~\ref{th:invariance} in the threshold case $\dl=0$. We first define the following set and solution map
	\[ \Sigma_0 :=\bigcup _{n\geq 1}\Sigma _{\frac1n},\qquad \Phi_0 (t):=\Phi _{\frac1n} (t)\quad \text{on $\Sigma_{\frac1n}$}, \ n \in \NB.\]
	From uniqueness of local-in-time solutions in $Z^{s,\frac12}_p$, $\Phi_0(t)$ is well-defined on $\Sigma_0$. Similarly, the flow property \eqref{flowproperty} extends to $\Phi_0(t)$ on $\Sigma_0$ from that of $\Phi_\frac1n(t)$ on $\Sigma_\frac1n$. In addition, $\Sigma_0$ is $\mu_0$-measurable since it is a countable union of Borel sets of $\FL^{s,p}(\T)$. Furthermore, the $\mu_0$-measurability of $\Phi_0(t)$ follows from its continuity on $\Sigma_0$ in the topology induced by $\FL^{s,p}(\T)$, the proof of which boils down to the argument for a fixed $\delta >0$ presented in Section~\ref{sec:inv}.
	
	It only remains to show that $\mu_0(\Sigma_0)= 1$ 
	and $\mu_0$ is invariant under the solution map $\Phi_0(t)$. 	
	For any $n\in\NB$, we have 
	\[ \mu _\frac1n (\Sigma_0 ^c)=\mu _\frac1n \Big( \bigcap _{n\geq 1}\Sigma _{\frac1n}^c\Big) \leq \mu _\frac1n (\Sigma_\frac1n ^c)=0,\]
and for any measurable $A\subset \Sigma_0$, $t\in\R$, $n\in\NB$, we have
	\begin{align*}
		\mu_\frac1n \big( \Phi_0(t)A \big) & = \mu_\frac1n \big( (\Phi_0(t)A) \cap \Sigma_\frac1n \big) = \mu_\frac1n \big( \Phi_\frac1n(t) (A\cap\Sigma_\frac1n) \big) = \mu_\frac1n (A \cap \Sigma_\frac1n) = \mu_\frac1n(A),
	\end{align*}
	from the invariance of $\mu_\frac1n$ under $\Phi_\frac1n(t)$ and the fact that $\Sigma_\frac1n$ has full $\mu_\frac1n$-measure. By taking the limit as $n\to\infty$ and using the convergence in \eqref{conv0}, we conclude that  $\mu _0(\Sigma_0 ^c)=0$ and that $\mu_0(\Phi_0(t)A) = \mu_0(A)$.
	
	This concludes the proof of Theorem~\ref{th:invariance} (b) in the case $k=5$ and $R = \|Q\|_{L^2(\R)}$. Although we have considered the mean zero case, the above argument also works for the problem without the mean zero condition (see Remark~\ref{rm:mean} and Appendix~\ref{ap:mean}). Moreover, the corresponding results for gKdV \eqref{gkdv} can be verified by the same argument as for the non-threshold case $R < \|Q\|_{L^2(\R)}$ (see the proof of Theorem~\ref{th:invariance_gkdv} in Section~\ref{sec:inv}).

	\section{Proof of Lemmas~\ref{lm:uniform} and \ref{lm:approx}} \label{ap:measure}
	
	First, we prove Lemma~\ref{lm:uniform} using Lemma~\ref{lm:tail}.
	\begin{proof}[Proof of Lemma~\ref{lm:uniform}]
		From Theorem~\ref{th:lwp}, we know that for $u_0 \in \FL^{s,p}(\T)$ with $\|u_0\|_{\FL^{s,p}} \leq K$, the corresponding solution $u_N$ to \eqref{truncated2} satisfies
		\begin{align*}
			\big\| u_N(t) \big\|_{\FL^{s,p}} \leq C_0 K,
		\end{align*}
		for $|t| \leq \dl \sim K^{-\gamma}$, $\gamma =\frac{k-1}{\theta}>0$ with $\theta$ given in Proposition~\ref{prop:nonlinear}, where $C_0>0$ does not depend on $s$. Note also that the constants can be taken uniformly in $N$. We want to establish a bound on $u_N(t)$ for all $|t| \leq T$. Let $[x]$ denote the integer part of a real number $x$ and define
		\begin{equation*}
			\Omega_N^s (T, \eps, A) = \bigcap_{j= - \big[\frac{T}{\dl}\big]}^{\big[\frac{T}{\dl}\big]} \Phi_N(j\dl) \Big( \Big\{ \|u_0\|_{\FL^{s,p}} \leq K \Big\}\Big),
		\end{equation*}
		where 
	$K=AC_s\big( \log \frac{T}{\eps}\big)^\frac12$ with a constant $C_s>0$ to be chosen later. 

		We start by showing (a). Let $B_K = \{ \|u_0\|_{\FL^{s,p}} \leq K\} $. From the uniqueness of solution to \eqref{truncated2} in each time interval $[j\dl, (j+1)\dl]$, we see that the solution map is invertible and
		$$\big[ \Phi_N(j\dl) (B_K) \big]^c = \Phi_N(j\dl)(B_K^c).  $$
		Consequently, 
		\begin{align*}
				\mu_N \big( [\Omega_N^s (T, \eps, A)]^c\big) &= \mu_N \bigg( \bigcup_{j= - \big[ \frac{T}{\dl}\big]}^{\big[\frac{T}{\dl}\big]} \Phi_N(j\dl) (B_K^c)  \bigg) \\&\leq \sum_{j= - \big[ \frac{T}{\dl}\big]}^{\big[\frac{T}{\dl}\big]} \mu_N \big( \Phi_N(j\dl) (B_K^c) \big) = 2 \Big[\frac{T}{\dl}\Big] \mu_N(B_K^c)
		\end{align*}
		from the invariance of $\mu_N$ under the flow $\Phi_N(t)$ of \eqref{truncated2} in Proposition~\ref{inv_trunc}. From Cauchy-Schwarz inequality, Lemma~\ref{lm:tail} and \eqref{FNR}, we have
		\begin{align*}
			\mu_N \big( [\Omega_N^s (T, \eps, A)]^c\big) \les \frac{T}{\dl} \int_{B_K^c} F_{N}(u) \ d\rho(u) \les \frac{T}{\dl} \|F_{N} \|_{L^2(d\rho)} \rho(B_K^c)^\frac12 \les \frac{T}{\dl} e^{-cK^2} \sim T K^\gamma e^{-cK^2}.
		\end{align*}
		Since $\log \frac{T}{\eps}\geq \log 2$ by the assumption, there exists $C_s>0$ such that if $K\geq C_s \big(\log \tfrac{T}{\eps}\big)^\frac12$, then $TK^\gamma e^{- c K^2 } \leq Te^{-\frac{c}{2} K^2}\ll \eps$. Hence, the above estimate, for $K=AC_s\big( \log \frac{T}{\eps}\big)^\frac12$ with $A\geq 1$ and such a constant $C_s$, ensures that $\mu_N \big( [\Omega_N^s (T, \eps, A)]^c\big) < \eps$, establishing (a).
		%
		With the invertibility of the solution map, (b) is a consequence of the local bound mentioned at the beginning, and (c) immediately follows from the definition of $\Omega_N^s (T, \eps, A)$. 
	\end{proof}
	
	Next, we derive Lemma~\ref{lm:approx} from the local theory.
	\begin{proof}[Proof of Lemma~\ref{lm:approx}]
		We only consider the positive time direction. We start by showing (a).
		Let $\mathcal{N}(u):=k\mathbf{P}(u^{k-1})\partial _xu$.
		By the local theory, with $\delta \sim (1+K)^{-\gamma}$ the solution $u_N$ to \eqref{truncated2} satisfies
		\begin{equation}\label{est:u_N}
			\| u_N\| _{Z^{\s,\frac12}_p([j\delta ,(j+1)\delta])}\leq C_2K,\qquad 0\leq j<[\tfrac{T}{\delta}]
		\end{equation}
		for some $C_2>0$.
		Note that the solution to \eqref{truncated2} in $C\big([-T,T];\F L^{\s,p}(\T)\big)$ coincides on each interval $[j\delta ,(j+1)\delta ]$ with the solution constructed by the iteration argument in $Z^{\s,\frac12}_p$, and also that
		\[ \mathbf{P}_{\leq N}u_N(t)=S(t-j\delta )\mathbf{P}_{\leq N}u_N(j\delta )+\int _{j\delta }^tS(t-t')\mathbf{P}_{\leq N}\mathcal{N}(\mathbf{P}_{\leq N}u_N(t'))\,dt', \quad t\in [j\delta ,(j+1)\delta ] \]
		for any $0\leq j<[\frac{T}{\delta}]$.
		We want to construct a solution $u$ to 
		\[ u(t)=S(t-j\delta )u(j\delta)+\int _{j\delta }^tS(t-t')\mathcal{N}(u(t'))\,dt', \qquad t\in [j\delta ,(j+1)\delta ]\]
		for each $j=0,1,\dots ,[\frac{T}{\delta}]-1$.
		This amounts to constructing $w(t):=u(t)-\mathbf{P}_{\leq N}u_N(t)$, which solves
		\begin{multline}\label{eq:w}
			w(t)=\Xi_j[w](t):=S(t-j\delta )w(j\delta )
			+\int _{j\delta}^tS(t-t')\mathbf{P}_{>N}\mathcal{N}(\mathbf{P}_{\leq N}u_N)(t')\,dt'
			\\
			+\int _{j\delta}^tS(t-t')\big\{ \mathcal{N}(w+\mathbf{P}_{\leq N}u_N)-\mathcal{N}(\mathbf{P}_{\leq N}u_N)\big\} (t')\,dt'.
		\end{multline}
		By the nonlinear estimates in $Z^{s, \frac12}_{p}$ and $Z^{\s,\frac12}_p$, together with \eqref{est:u_N}, we have
		\begin{align*}
			&\| \Xi _j[w]\| _{Z^{s, \frac12}_p([j\delta ,(j+1)\delta ])}\\
			&\quad \leq C_0\| w(j\delta )\| _{\F L^{s ,p}}+C_1\delta ^\theta \Big( \| w\| _{Z^{s, \frac12}_p([j\delta ,(j+1)\delta ])} +C_2K\Big) ^{k-1}\| w\| _{Z^{s, \frac12}_p([j\delta ,(j+1)\delta ])} \\
			&\qquad +C_1N^{-(\s-s)}\delta ^\theta (C_2K)^k,\\
			&\| \Xi _j[w]-\Xi _j[\tilde{w}]\| _{Z^{s, \frac12}_p([j\delta ,(j+1)\delta ])}\\
			&\quad \leq C_1\delta ^\theta \Big( \| w\| _{Z^{s, \frac12}_p([j\delta ,(j+1)\delta ])} +\| \tilde{w}\| _{Z^{s, \frac12}_p([j\delta ,(j+1)\delta ])} +C_2K\Big) ^{k-1}\| w-\tilde{w}\| _{Z^{s, \frac12}_p([j\delta ,(j+1)\delta ])} ,
		\end{align*}
		for some $C_0>0$ and $C_1=C_1(s,p)>0$. 
		Therefore, taking smaller $\delta \sim_{s,p} (1+K)^{-\gamma}$ if necessary
, we can show that $\Xi _j$ is a contraction on 
		\[ \big\{ w\in Z^{s, \frac12}_p([j\delta ,(j+1)\delta ]) : \| w\| _{Z^{s, \frac12}_p([j\delta ,(j+1)\delta ])}\leq 2C_0\| w(j\delta )\| _{\F L^{s ,p}}+N^{-(\s-s)}K \big\} \]
		as long as
		\[ \| w(j\delta )\| _{\F L^{s ,p}}\leq K.\]
		Starting from $\| w(0)\| _{\F L^{s ,p}}\leq N^{-(\s-s)}K$, we obtain the solution $w$ to \eqref{eq:w} on $[j\delta ,(j+1)\delta ]$ with 
		\[ \| w((j+1)\delta )\| _{\F L^{s ,p}}\leq \tilde{C}_0^{j+1}N^{-(\s- s)}K,\qquad j=0,1,\dots, \Big[ \frac{T}{\dl}\Big]-1 ,\]
		for some $\tilde{C}_0>0$.
		In particular, the solution can be extended up to $t=T$ if $N$ satisfies
		\[ N^{\s-s} \geq e^{C_3(1+K)^{\gamma}T}~(\geq \tilde{C}_0^{[\frac{T}{\delta}]})\]
		for some $C_3=C_3(s,p)>0$.
		Consequently, for $N$ large enough, we obtain
		\[ \max _{0\leq t\leq T}\| w(t)\| _{\F L^{s ,p}}\leq \max _{0\leq j<[\frac{T}{\delta }]}\tilde{C}_0^{j+1}N^{-(\s - s)}K\leq e^{C_3(1+K)^{\gamma}T}N^{-(\s -s )}K.\]
		The estimate follows by further imposing $N \geq N_0$ where $N_0 \sim \exp\big( \frac{CK^\gamma T}{\s -s} \big)$.
		
		To establish (b), note that we can also write $w(t)$ as follows
		\begin{multline*}
			w(t)=\tilde{\Xi}_j[w](t):=S(t-j\delta )w(j\delta )
			+\int _{j\delta}^tS(t-t')\mathbf{P}_{>N}\mathcal{N}(u)(t')\,dt'
			\\
			+\int _{j\delta}^tS(t-t') \PP_{\leq N}\big\{ \mathcal{N}(u)-\mathcal{N}(u-w)\big\} (t')\,dt'.
		\end{multline*}
	The estimate then follows from the same arguments as for (a).
	\end{proof}
	
\end{appendix}

\begin{ack}\rm 
	The authors would like to thank Tadahiro Oh for suggesting the problem and for his continuous support throughout the work. 
	They are also grateful to Terence Tao and Rowan Killip for giving them various constructive suggestions on how to establish the invariance of the Gibbs measure for the original gKdV equation from that for the gauged one, which is one of the key steps in establishing the main result.
	In addition, they deeply appreciate a number of  comments given by the anonymous referees, which are most useful for ensuring the correctness and enhancing the intelligibility of certain arguments.
	A.C. was supported by the Maxwell Institute Graduate School in Analysis and its Applications, a Centre for Doctoral Training funded by the UK Engineering and Physical Sciences Research Council (grant EP/L016508/01), the Scottish Funding Council, Heriot-Watt University and the University of Edinburgh. N.K. was supported by JSPS (Grant-in-Aid for Young Researchers (B) no.~16K17626). The authors also acknowledge support from the European Research Council (grant no. 637995 ``ProbDynDispEq'' and grant no.~864138 ``SingStochDispDyn'').

\end{ack}


\end{document}